\documentclass[review,onefignum,onetabnum]{siamart190516}
\usepackage{subfigure}
\usepackage{nicefrac}
\newtheorem{example}{Example}[section]
\newcommand\dt {{\Delta t}}
\newcommand{\hf}{\nicefrac{1}{2}}
\newcommand{\Dh}{\Delta_h}
\newcommand{\nabh}{\nabla_{\! h}}
\newcommand{\nrm}[1]{\left\| #1 \right\|}
\newcommand{\ciptwo}[2]{\left\langle #1 , #2 \right\rangle_\Omega}
\newcommand{\cipgen}[3]{\left\langle #1 , #2 \right\rangle_{#3}}
\newcommand{\eipx}[2]{\left[ #1 , #2 \right]_{\rm x}}
\newcommand{\eipy}[2]{\left[ #1 , #2 \right]_{\rm y}}
\newcommand{\eipvec}[2]{\left[ #1 , #2 \right]_{\Omega}}
\newcommand{\sumi}{\sum_{\ell = 1}^9}
\newcommand{\sumij}{\sum_{i,j=1}^N}
\usepackage{lipsum}
\usepackage{amsfonts}
\usepackage{graphicx}
\usepackage{epstopdf}
\usepackage{algorithmic}
\ifpdf
  \DeclareGraphicsExtensions{.eps,.pdf,.png,.jpg}
\else
  \DeclareGraphicsExtensions{.eps}
\fi

\newsiamremark{remark}{Remark}
\newsiamremark{hypothesis}{Hypothesis}
\crefname{hypothesis}{Hypothesis}{Hypotheses}
\newsiamthm{claim}{Claim}

\headers{A positive, energy stable scheme for the ternary MMC system}{L. Dong, C. Wang, S. M. Wise, and Z. Zhang}
\title{A positivity-preserving, energy stable scheme for a Ternary Cahn-Hilliard system with the singular interfacial parameters} 

\author{Lixiu Dong\thanks{College of Education for the Future, Beijing Normal University at Zhuhai, Guangdong 519087, P.R. China (\email{lxdong@mail.bnu.edu.cn}).}
\and Cheng Wang\thanks{Department of Mathematics, The University of Massachusetts, North Dartmouth, MA 02747, USA (corresponding author: \email{cwang1@umassd.edu}).}
\and Steven M. Wise\thanks{Department of Mathematics, The University of Tennessee, Knoxville, TN 37996, USA (\email{swise1@utk.edu}).}
\and Zhengru Zhang\thanks{Laboratory of Mathematics and Complex Systems, Ministry of Education and School of Mathematical Sciences, Beijing Normal University, Beijing 100875, P.R. China (\email{zrzhang@bnu.edu.cn}).}}
\usepackage{amsopn}

\begin{document} 

\maketitle

\begin{abstract}
In this paper, we construct and analyze a uniquely solvable, positivity preserving and unconditionally energy stable finite-difference scheme for the periodic three-component Macromolecular Microsphere Composite (MMC) hydrogels system, a ternary Cahn-Hilliard system with a Flory-Huggins-deGennes free energy potential. The proposed scheme is based on a convex-concave decomposition of the given energy functional with two variables, and the centered difference method is adopted in space. We provide a theoretical justification that this numerical scheme has a pair of unique solutions, such that the positivity is always preserved for all the singular terms, i.e., not only two phase variables are always between $0$ and $1$, but also the sum of two phase variables is between $0$ and $1$, at a point-wise level. In addition, we use the local Newton approximation and multigrid method to solve this nonlinear numerical scheme, and various numerical results are presented, including the numerical convergence test, positivity-preserving property test, energy dissipation and mass conservation properties.
\end{abstract}




\section{Introduction}\label{sec-intro}
\setcounter{equation}{0}

Macromolecular microsphere composite (MMC) hydrogels, a class of polymeric materials, have attracted theoretical and experimental studies due to their well-defined network microstructures and high mechanical strength. Various methods have been developed to model the evolution of MMC hydrogels. In \cite{Zhai2012Investigation}, the authors presented a binary mathematical model to describe the periodic structures and the phase transitions of the MMC hydrogels based on Boltzmann entropy theory. Their model, the MMC-TDGL equation, is structurally similar to the Cahn-Hilliard equation. Most existing works for the MMC-TDGL equation are based on the two-phase model; see \cite{Dong2019a, Lixiao2015, Li2016An, Lixiao2}, et cetera. 

The Allen-Cahn and Cahn-Hilliard equations are well-known gradient flows with either polynomial Ginzburg-Landau or singular Flory-Huggins-type free energy density. These equations model spinodal decomposition and phase separation in a two-phase fluid in either the non-conserved or conserved setting, respectively. There have been many theoretical analyses and numerical approximations for these kinds of gradient flows in the two-phase case~\cite{chen19a, chen16, chenY18, cheng2019a, cheng16a, diegel15a, diegel17, diegel16, fengW18b, han15, liuY17, Yan18}. For the ternary Cahn-Hilliard system, the general framework is to adopt three independent phase variables $(\phi_1,\phi_2,\phi_3)$ while enforcing a mass conservation (or ``no-voids") constraint $\phi_1+\phi_2+\phi_3=1$. See the related works~\cite{Boyer2006Numerical, Boyer2011Numerical, Yang2017Numerical}.

In this work, we consider a ternary time-dependent Ginzburg-Landau mesoscopic model with a given coarse-grained free energy, which is an improvement in some ways over the model proposed in~\cite{Zhai2012Investigation}, as it removes the assumption that the number of the graft chain around a large ball is proportional to the perimeter in the modeling process. For this ternary Cahn-Hilliard system, the following singular energy potential is taken into consideration:
	\begin{equation}
G_o(\phi_1,\phi_2,\phi_3) = \int_{\Omega} \left\{ S_o(\phi_1,\phi_2,\phi_3)+ \frac{1}{36} \sum_{i=1}^3 \frac{\varepsilon_i^2}{\phi_i} |\nabla\phi_i|^2 + H_o(\phi_1,\phi_2,\phi_3)\right\} d{\bf x},
	\label{continuous energy}
	\end{equation}
where $S_o(\phi_1,\phi_2,\phi_3)+ H_o(\phi_1,\phi_2,\phi_3)$ is the reticular free energy density:
	\begin{align*}
 S_o(\phi_1,\phi_2,\phi_3) & = \frac{\phi_1}{M_0} \ln \frac{\alpha \phi_1}{M_0}+\frac{\phi_2}{N_0} \ln \frac{\beta \phi_2}{N_0} + \phi_3 \ln \phi_3,
 	\\
 H_o(\phi_1,\phi_2,\phi_3) & = \chi_{12} \phi_1 \phi_2 + \chi_{13} \phi_1 \phi_3 + \chi_{23}\phi_2 \phi_3.
	\end{align*}
The term $S_o$ is often called the ideal solution part, and $H_o$ stands for the entropy of mixing part. The sum $S_o+ H_o$ is also called the regular solution model in material science and the Flory-Huggins model in polymer chemistry. The domain $\Omega \subset \mathbb{R}^2$ is assumed to be open, bounded, and simply connected. We focus on the 2-D case for simplicity of presentation, while an extension to the 3-D gradient flow is straightforward. The unknown phase variable $\phi_1$, $\phi_2$ and $\phi_3$ are conserved field variables, representing the concentration of the macromolecular microsphere, the concentration of the polymer chain, and the concentration of the solvent, respectively. These three phase variables are subject to the ``no-voids" constraint $\phi_1 + \phi_2 + \phi_3= 1$. We denote by $M_0$ the relative volume of one macromolecular microsphere, and by $N_0$ the degree of polymerization of the polymer chains. The coefficient $\varepsilon_i$ is called the statistical segment length of the $i$-th component, which is always positive. The parameters $\alpha$ and $\beta$ depend on $M_0$ and $N_0$:
   \[
\alpha= \pi\left(\sqrt{\frac{M_0}{\pi}}+ \frac{N_0}{2}\right)^2,\quad \beta=
2\sqrt{\frac{M_0}{\pi}} + N_0.
   \]

By $\chi_{12}, \chi_{13}$, and $\chi_{23}$ we denote the Huggins interaction parameters between (i) the macromolecular microspheres and polymer chains, (ii) the macromolecular microspheres and solvent, and (iii) the polymer chains and solvent, respectively. All these parameters are positive. In this paper, we choose parameters satisfying the inequality
	\[
4\chi_{13}\chi_{23}-(\chi_{12}-\chi_{13}-\chi_{23})^2 >0 ,
	\]
which guarantees the concavity of the entropy of mixing $H_0$ term, as we shall see.

Making use of the no-voids constraint $\phi_3 = 1 - \phi_1 - \phi_2$, we can rewrite the energy functional as
	\begin{align}
	\label{continuous energy1}
G(\phi_1,\phi_2) &= \int_{\Omega} \bigg\{ S(\phi_1,\phi_2)+ \frac{\varepsilon_1^2|\nabla\phi_1|^2}{36\phi_1} + \frac{\varepsilon_2^2|\nabla\phi_2|^2}{36\phi_2}+
\frac{\varepsilon_3^2|\nabla(1-\phi_1-\phi_2)|^2}{36(1-\phi_1-\phi_2)}
	\\
& \quad + H(\phi_1,\phi_2) \bigg\} d{\bf x},	
	\nonumber
	\end{align}
where, naturally,
	\begin{align*}
 S(\phi_1,\phi_2) & = \frac{\phi_1}{M_0} \ln \frac{\alpha \phi_1}{M_0}+\frac{\phi_2}{N_0} \ln \frac{\beta \phi_2}{N_0} + (1-\phi_1-\phi_2) \ln (1-\phi_1-\phi_2),
 	\\
 H(\phi_1,\phi_2) &= \chi_{12} \phi_1 \phi_2 + \chi_{13} \phi_1 (1-\phi_1-\phi_2) + \chi_{23}\phi_2 (1-\phi_1-\phi_2).
     \end{align*}
The ternary MMC dynamic equations become the $H^{-1}$ gradient flows associated with the given energy functional \eqref{continuous energy1}:
     \begin{equation}\label{MMC3term-equation}
\partial_t \phi_1 = \mathcal{M}_1 \Delta \mu_1,\quad \partial_t \phi_2 = \mathcal{M}_2 \Delta
\mu_2,
     \end{equation}
 where $\mathcal{M}_1, \mathcal{M}_2 >0$ are mobilities, which are assumed to be positive constants. The terms $\mu_1$ and $\mu_2$ are the chemical potentials with respect to $\phi_1$ and $\phi_2$,
 respectively, i.e.,
     \begin{align}
\mu_1:=\delta_{\phi_1}G & = \frac{1}{M_0}\ln \frac{\alpha \phi_1}{M_0} - \ln(1-\phi_1-\phi_2) - 2 \chi_{13} \phi_1 +(\chi_{12}-\chi_{13}-\chi_{23})\phi_2
	\nonumber
	\\
& \quad  + \chi_{13} + \frac{1}{M_0} -1 - \frac{\varepsilon_1^2|\nabla\phi_1|^2}{36\phi_1^2} - \nabla \cdot \left(\frac{\varepsilon_1^2\nabla \phi_1}{18\phi_1}\right)
     \label{mu_1}
	\\
& \quad + \frac{\varepsilon_3^2|\nabla(1-\phi_1-\phi_2)|^2}{36(1-\phi_1-\phi_2)^2}
   + \nabla \cdot \left(\frac{\varepsilon_3^2 \nabla (1-\phi_1-\phi_2)}{18(1-\phi_1-\phi_2)}
     \right),
	\nonumber
	\\
\mu_2:=\delta_{\phi_2}G & = \frac{1}{N_0}\ln \frac{\beta \phi_2}{N_0} - \ln(1-\phi_1-\phi_2) - 2 \chi_{23} \phi_2 +(\chi_{12}-\chi_{13}-\chi_{23})\phi_1
	\nonumber
	\\
& \quad  + \chi_{23} + \frac{1}{N_0} -1  - \frac{\varepsilon_2^2|\nabla\phi_2|^2}{36\phi_2^2} - \nabla \cdot \left(\frac{\varepsilon_2^2\nabla \phi_2}{18\phi_2}\right)
	\label{mu_2}
	\\
& \quad  + \frac{\varepsilon_3^2|\nabla(1-\phi_1-\phi_2)|^2}{36(1-\phi_1-\phi_2)^2} + \nabla \cdot \left(\frac{\varepsilon_3^2 \nabla (1-\phi_1-\phi_2)}{18(1-\phi_1-\phi_2)} \right).
     \nonumber
     \end{align}
For simplicity,
we assume that periodic boundary conditions hold. It is then easy to see that the energy is non-increasing for the ternary MMC model. The evolution equations \eqref{MMC3term-equation} are mass conservative; the mass fluxes are proportional to the gradients of the respective chemical potentials.

Concerning the ternary Cahn-Hilliard type model with polynomial Ginzburg-Landau free energy density potential and constant surface diffusion coefficients, there have been quite a few existing numerical works to address the issue of energy stability. For example, several finite element schemes have been studied in an earlier work~\cite{Boyer2011Numerical}, based on different semi-implicit methods to pursue an energy dissipation property. Recently, a Fourier pseudo-spectral numerical scheme was constructed in~\cite{ChenTernaryCH}, based on a non-standard convex-concave decomposition of the physical energy; the unique solvability and unconditional energy stability of the corresponding numerical scheme were established at a theoretical level. Besides the convex splitting approach, an invariant energy quadrant (IEQ) algorithm was designed in~\cite{Yang2017Numerical}. Therein a stability analysis was  proved for a numerically modified energy, not for the original energy functional.

By comparison, the ternary Cahn-Hilliard system~\eqref{MMC3term-equation} -- \eqref{mu_2} is much more difficult than the versions mentioned above. Due to the singular nature of the Flory-Huggins logarithmic free energy density, the positivity-preserving property has to be enforced for the numerical solution to make the scheme well-defined, which turns out to be a very challenging issue. For example, an application of either the invariant energy quadrant (IEQ)~\cite{guillen14}, scalar auxiliary variable (SAV)~\cite{shen18b, shen18a} or linear stabilization method~\cite{LiD2017, LiD2016a} would not be able to enforce such a property, due to the explicit treatment of the nonlinear singular terms. In fact, an extension of the singular energy functional (beyond the singular limit values) has to be made to define the corresponding linear numerical schemes. In addition to the difficulty associated with the positivity-preserving behavior of the numerical solution, the highly nonlinear and singular nature of the surface diffusion coefficients makes the system even more challenging, at both the analytic and numerical levels. In this paper, we propose and analyze a numerical scheme for the ternary MMC hydrogels system~\eqref{MMC3term-equation} -- \eqref{mu_2}, with three theoretical properties justified: positivity-preserving, unique solvability, and unconditional energy stability. This scheme is based on the convex-concave decomposition of the original energy functional, which turns out to be highly non-trivial even for the polynomial approximation one~\cite{ChenTernaryCH}, due to the multi phase variables involved. 
In order to apply the framework of such a decomposition for the terms involved with multi phase variables, a careful calculation of the Hessian matrix has to be performed. As analyzed in a recent article~\cite{chen19b} for the Flory-Huggins Cahn-Hilliard flow with constant surface diffusion coefficient, an implicit treatment of the nonlinear singular logarithmic term is necessary to theoretically justify its positivity-preserving property. In addition to the logarithmic terms, the chemical potential expansions with the nonlinear deGennes surface diffusion energy have to be implicitly updated in the numerical scheme, because of its convex nature in terms of all the phase variables. This leads to a highly nonlinear, highly singular numerical system, while the linear expansive term is treated explicitly. However, a more careful analysis reveals that, the convex and the singular natures of these implicit nonlinear parts prevent the numerical solutions approach the singular limit values of $0$ and $1$, so that the positivity-preserving property is available for all the phase variables. Such a theoretical justification is much more complicated than the one with constant surface diffusion coefficient case, as reported in~\cite{chen19b}, because of the mixed terms involved in the nonlinear surface diffusion part. With the positivity property justified, the unique solvability becomes a direct consequence of the convexity associated with the implicit terms in the numerical algorithm. An unconditional energy stability could also be derived using a convexity argument. 

The rest part of this paper is organized as follows. In \Cref{sec:existence of a convex-concave decomposition}, we show a convex-concave decomposition of the energy~\eqref{continuous energy1}. In \Cref{sec:numerical scheme},
we present a finite difference scheme based on a convex splitting of the energy functional.
In \Cref{sec:positivity-preserving property}, the unique solvability and the
positivity preserving property of the numerical solutions are analyzed. The unconditional
energy stability analysis is provided in \Cref{sec:unconditional energy stability}.
Various numerical results are presented in \Cref{sec:numerical results}. Finally, we
give some concluding remarks in \Cref{sec:conclusions}.
%

	
	\section{Existence of a convex-concave decomposition}
	\label{sec:existence of a convex-concave decomposition}
	\setcounter{equation}{0}
	
In this section, we will give a convex-concave decomposition of the energy~\eqref{continuous energy1}. The following preliminary results are needed.
	\begin{proposition}
	\label{convex splitting}
Define the functions
	\[
T_1(u,v):=\frac{v^2}{36 u}, \quad u\in (0,\infty), \quad v\in\mathbb{R};
	\]
	\[
T_2(u_1,u_2,v_1,v_2):=\frac{(v_1+v_2)^2}{36(1-u_1-u_2)},\quad u_1, u_2, v_1,v_2\in\mathbb{R};
	\]
	\[
T_3(u,v,w):=\frac{w^2}{36(u+v)}, \quad u,v,w\in\mathbb{R}.
	\]
	\begin{enumerate}
	\item	
$T_1(u,v)$ is convex in $(0, +\infty)\times \mathbb{R}$.
	\item
$T_2(u_1,u_2,v_1,v_2)$ is convex in $\mathbb{R}^4$, provided that $u_1+u_2<1$.
	\item
$T_3(u,v,w)$ is convex in $\mathbb{R}^3$, provided that $u+v>0$.
	\item
$S(u_1,u_2)$ is convex in the Gibbs Triangle, ${\mathcal G}$, defined as
	\[
{\mathcal G} := \left\{(u_1,u_2) \ | u_1,u_2 > 0, \  u_1+u_2 < 1 \right\}.	
	\]
	\item
$H(u_1,u_2)$ is concave, provided that $4\chi_{13}\chi_{23}-(\chi_{12}-\chi_{13}-\chi_{23})^2 > 0$.
	\end{enumerate}

	\end{proposition}

	\begin{proof}
(1) For $T_1(u,v)$, a careful calculation gives its Hessian matrix:
     \begin{equation}
\mathsf{H}_1 = \frac{1}{36}\left(
	\begin{array}{cc}
\frac{2v^2}{u^3} & -\frac{2v}{u^2}
	\\
-\frac{2v}{u^2} & \frac{2}{u}
     \end{array}
\right).\nonumber
     \end{equation}
The first-order principal minors of the matrix $\mathsf{H}_1$ are given by: $D_1 =\frac{v^2}{18 u^3}$, $D_2 = \frac{1}{18 u}$, which are both non-negative when $u \in (0, +\infty)$ and $v \in \mathbb{R}$. In addition, the second-order principal minor becomes $D_{12} = 0$.  Therefore, we conclude that the Hessian Matrix $\mathsf{H}_1$ is positive semi-definite and thus $T_1$ is convex in $(0, \infty)\times \mathbb{R}$.

(2) The Hessian matrix for $T_2(u_1,u_2,v_1,v_2)$ turns out to be 
	\begin{equation*}
\mathsf{H}_2 = \frac{1}{36}\left(
	\begin{array}{cccc}
\frac{2A^2}{B^3} & \frac{2A^2}{B^3} & \frac{2A}{B^2} & \frac{2A}{B^2}
	\\
\frac{2A^2}{B^3} & \frac{2A^2}{B^3} & \frac{2A}{B^2} & \frac{2A}{B^2}
	\\
\frac{2A}{B^2} & \frac{2A}{B^2} & \frac{2}{B} & \frac{2}{B}
	\\
\frac{2A}{B^2} & \frac{2A}{B^2} & \frac{2}{B} & \frac{2}{B}
	\\
     \end{array}
\right),  \quad A = v_1 + v_2 , \quad B = 1 - u_1 - u_2 .
     \end{equation*}
The first-order principal minors of the matrix $\mathsf{H}_2$ are $D_1 = D_2 = \frac{A^2}{18 B^3}$, $D_3 = D_4 = \frac{1}{18 B}$, which are positive values. Meanwhile, all other principal minors are equal to 0. In general, all these principal minors are non-negative when $u_1+u_2<1$. Therefore, we conclude that the Hessian Matrix $\mathsf{H}_2$ is positive semi-definite and thus $T_2$ is convex when $u_1+u_2<1$.

(3) For $T_3 (u,v,w)$, the Hessian matrix has the following form:
\begin{equation}
     \mathsf{H}_3 =
     \frac{1}{36}\left(
     \begin{array}{ccc}
     \frac{2w^2}{(u+v)^3} & \frac{2w^2}{(u+v)^3}&\frac{-2w}{(u+v)^2}\\
     \frac{2w^2}{(u+v)^3} & \frac{2w^2}{(u+v)^3}&\frac{-2w}{(u+v)^2}\\
     \frac{-2w}{(u+v)^2} & \frac{-2w}{(u+v)^2}&\frac{2}{u+v}\\
     \end{array}
     \right).
     \end{equation}
The first-order principal minors of the matrix $\mathsf{H}_3$ are $D_1 = \frac{w^2}{18(u+v)^3}$,  $D_2 = \frac{w^2}{18(u+v)^3}$, $D_3=\frac{1}{18(u+v)}$, which are positive values. Again, all other principal minors are equal to 0. All these principal minors are non-negative when $u+v>0$. Then we conclude that the Hessian Matrix $\mathsf{H}_3$ is positive semi-definite and thus $T_3$ is convex when $u+v>0$.

(4) For $S(u_1,u_2)$, the Hessian matrix is
     \begin{equation}
\mathsf{H}_S = \left(
    \begin{array}{cc}
\frac{1}{M_0 u_1} + \frac{1}{1-u_1-u_2} & \frac{1}{1-u_1-u_2}
	\\
\frac{1}{1-u_1-u_2} &\frac{1}{N_0 u_2} + \frac{1}{1-u_1-u_2}
    \end{array}
\right).
    \end{equation}
The first-order principal minors of the matrix $\mathsf{H}_S$ are given by $D_1 = \frac{1}{M_0 u_1} + \frac{1}{1-u_1-u_2}$, $D_2 = \frac{1}{N_0 u_2} + \frac{1}{1-u_1-u_2}$, which are positive values. The second-order principal minor is determined as
     \[
D_{12} = \det(\mathsf{H}_S)= \frac{1}{M_0 N_0 u_1u_2 } + \frac{1}{M_0 u_1} + \frac{1}{N_0 u_2} +\frac{1}{1-u_1-u_2}.
     \]
All these principal minors are positive when $u_1, u_2 \in (0, +\infty)$ and $u_1+u_2<1$. Consequently, the Hessian matrix $\mathsf{H}_S$ is positive-definite and thus $S$ is convex in the Gibbs triangle $\mathcal{G}$.

(5) The Hessian matrix of $H (u_1, u_2)$ becomes 
	\begin{equation}
\mathsf{H}_H= \left(
	\begin{array}{cc}
-2\chi_{13}  &  \chi_{12}-\chi_{13}-\chi_{23}
	\\
\chi_{12}-\chi_{13}-\chi_{23} &  -2\chi_{23}
	\\
	\end{array}
\right).
	\end{equation}
The first-order principal minors of $\mathsf{H}_H$ are given by $D_1 = -2\chi_{13} < 0$, $D_2 = -2\chi_{13} < 0$. In addition, the second-order principal minor of $\mathsf{H}_H$ becomes
     \[
D_{12} = \det(\mathsf{H}_H)= 4\chi_{13}\chi_{23}-(\chi_{12}-\chi_{13}-\chi_{23})^2 > 0.
     \]
Therefore, the Hessian matrix $\mathsf{H}_H$ is negative-definite and thus $H$ is concave when $4\chi_{13}\chi_{23}-(\chi_{12}-\chi_{13}-\chi_{23})^2 > 0$.
\end{proof}

	\begin{lemma}[Existence of a convex-concave decomposition]
	\label{lemma-1}
Assume that $\phi_1,\phi_2: \Omega \rightarrow (0, 1)$ are periodic and sufficiently
regular, with point values in the Gibbs Triangle, $\mathcal{G}$. The functionals
     \begin{align}
G_c(\phi_1,\phi_2) &:= \int_{\Omega} S(\phi_1,\phi_2)+ \frac{\varepsilon_1^2|\nabla\phi_1|^2}{36\phi_1} + \frac{\varepsilon_2^2|\nabla\phi_2|^2}{36\phi_2}+ \frac{\varepsilon_3^2|\nabla(1-\phi_1-\phi_2)|^2}{36(1-\phi_1-\phi_2)} d {\bf x},
	\label{convex-energy-c}
	\\
G_e(\phi_1,\phi_2) & := -\int_{\Omega} H(\phi_1,\phi_2) d {\bf x} .
 \label{convex-energy-e}
     \end{align}
are convex. Therefore, $G(\phi_1,\phi_2)=G_c(\phi_1,\phi_2)-G_e(\phi_1,\phi_2)$ is a convex-concave decomposition of the energy.
	\end{lemma}

	\begin{proof}
The fact that $G(\phi_1,\phi_2)=G_c(\phi_1,\phi_2)-G_e(\phi_1,\phi_2)$ is obvious. Suppose that
	\[
(u_1,u_2) \in {\mathcal G} = \left\{(u_1,u_2) \ | u_1,u_2 > 0, \  u_1+u_2 < 1 \right\}
	\]
and set $\vec{u} : =(u_1,u_2,u_3,u_4,u_5,u_6)\in \mathcal{G}\times\mathbb
{R}^4$. Define
      \begin{align*}
e_c(\vec{u}) & : = S(u_1,u_2) + \varepsilon_1^2 T_1(u_1,u_3) + \varepsilon_1^2 T_1(u_1,u_5) + \varepsilon_2^2 T_1(u_2,u_4)
	\\
& \quad + \varepsilon_2^2 T_1(u_2,u_6) + \varepsilon_3^2 T_2(u_1,u_2,u_3,u_4) + \varepsilon_3^2 T_2(u_1,u_2,u_5,u_6),
	\\
e_e(\vec{u}) &:= -H(u_1,u_2).
      \end{align*}
\Cref{convex splitting} suggests that $e_c$ and $e_e$ are convex in $\mathcal{G}\times \mathbb {R}^4$. Therefore, we have the following inequality according to the definition of a convex function:
$\forall \, \lambda \in (0,1),  \vec{u},\vec{v} \in \mathcal{G}\times\mathbb {R}^4$,
     \begin{equation}\label{e_c}
e_c(\lambda \vec{u}+(1-\lambda)\vec{v})\leq \lambda e_c(\vec{u})+(1-\lambda)e_c(\vec{v}).
     \end{equation}
It is noticed that
     \begin{align*}
G_c(\phi_1,\phi_2) &= \int_{\Omega}
e_c(\phi_1,\phi_2,{\phi_1}_{x},{\phi_2}_{x},{\phi_1}_{y},{\phi_2}_{y}) d {\bf x},
	\\
G_e(\phi_1,\phi_2) & = \int_{\Omega} e_e(\phi_1,\phi_2) d {\bf x}.
     \end{align*}
Setting $\vec{u} :=(\phi_1,\phi_2,{\phi_1}_{x},{\phi_2}_{x},{\phi_1}_{y},{\phi_2}_{y})$ and $\vec{v}=(\psi_1,\psi_2,{\psi_1}_{x},{\psi_2}_{x},{\psi_1}_{y},{\psi_2}_{y})$, and integrating inequality $\eqref{e_c}$ leads to
     \[
G_c(\lambda \phi_1 +(1-\lambda)\psi_1,\lambda \phi_2 +(1-\lambda)\psi_2)\leq \lambda
G_c(\phi_1,\phi_2)+(1-\lambda)G_c(\psi_1,\psi_2),
     \]
which indicates that $G_c(\phi_1,\phi_2)$ is a convex functional of $\phi_1$ and $\phi_2$. Using a similar argument, we see that $G_e(\phi_1,\phi_2)$ is also convex.
	\end{proof}

The following estimate is the foundation of the energy stability. The proof, which is practically the same as that in~\cite{wise09a}, is independent on the specific form of $G(\phi_1,\phi_2)$.

	\begin{lemma}
Suppose that $\Omega=(0,L_x)\times(0,L_y)$ and $(\phi_1,\phi_2),(\psi_1,\psi_2):\Omega \rightarrow \mathcal{G}$ are periodic and sufficiently regular. Consider the canonical convex splitting of the energy $G(\phi_1,\phi_2)$ in \eqref{continuous energy1} into $G=G_c-G_e$ given in \eqref{convex-energy-c} -- \eqref{convex-energy-e}. Then
     \begin{align}
G(\vec{\phi})-G(\vec{\psi}) & \leq (\delta_{\phi_1} G_c(\vec{\phi})-\delta_{\phi_1} G_e(\vec{\psi}),\phi_1-\psi_1)_{L^2}
	\label{Semi-energy-inequality}
	\\
& \quad +(\delta_{\phi_2} G_c(\vec{\phi})-\delta_{\phi_2} G_e(\vec{\psi}),\phi_2-\psi_2)_{L^2},
	\nonumber
	\end{align}
where $\vec{\phi}=(\phi_1,\phi_2)$, $\vec{\psi}=(\psi_1,\psi_2)$.
\end{lemma}
\begin{proof}
Set
	\[
G_c(\vec{\phi}) = \int_{\Omega} e_c(\phi_1,\phi_2,{\phi_1}_{x},{\phi_2}_{x},{\phi_1}_{y},{\phi_2}_{y}) d
{\bf x}.
	\]
If $(\phi_1,\phi_2)\in\mathcal{G}$, \Cref{lemma-1} ensures the convexity of $e_c(\vec{u})$ in $\mathcal{G}\times \mathbb{R}^4$. We have the equivalent statement
     \[
e_c(\vec{v})-e_c(\vec{u}) \geq \nabla_{\vec{u}} e_c(\vec{u})\cdot (\vec{v}-\vec{u}),
     \]
for any $\vec{u},\vec{v}\in \mathcal{G}\times \mathbb {R}^4$.

Now setting
	\[
\vec{u}=(\phi_1,\phi_2,{\phi_1}_{x},{\phi_2}_{x},{\phi_1}_{y},{\phi_2}_{y}), \quad \vec{v}=(\psi_1,\psi_2,{\psi_1}_{x},{\psi_2}_{x},{\psi_1}_{y},{\psi_2}_{y}) , 
	\]
and integrating-by-parts, we get the inequality
     \begin{equation}
G_c(\vec{\phi})-G_c(\vec{\psi}) \geq (\delta_{\phi_1} G_c(\vec{\psi}),\phi_1-\psi_1)_{L^2}
+(\delta_{\phi_2} G_c(\vec{\psi}),\phi_2-\psi_2)_{L^2}.\label{G_c_inequality}
     \end{equation}
By a similar analysis for $G_e$, we see that 
     \begin{equation}
G_e(\vec{\psi})-G_e(\vec{\phi}) \geq (\delta_{\phi_1} G_e(\vec{\phi}),\psi_1-\phi_1)_{L^2}
 +(\delta_{\phi_2} G_e(\vec{\phi}),\psi_2-\phi_2)_{L^2}.\label{G_e_inequality}
     \end{equation}
Adding \eqref{G_c_inequality} and \eqref{G_e_inequality} yields
	\begin{align*}
& \hspace{-0.25in}G(\vec{\phi})-G(\vec{\psi})
	\\
& =\left(G_c(\vec{\phi})-G_c(\vec{\psi})\right)-\left(G_e(\vec{\phi})-G_e(\vec{\psi})\right)
	\\
& \leq (\delta_{\phi_1} G_c(\vec{\phi}),\phi_1-\psi_1)_{L^2} +(\delta_{\phi_2} G_c(\vec{\phi}),\phi_2-\psi_2)_{L^2}
	 \\
& \quad -\left((\delta_{\phi_1} G_e(\vec{\psi}),\phi_1-\psi_1)_{L^2} +(\delta_{\phi_2} G_e(\vec{\psi}),\phi_2-\psi_2)_{L^2}\right)
	\\
& =(\delta_{\phi_1} G_c(\vec{\phi})-\delta_{\phi_1}
G_e(\vec{\psi}),\phi_1-\psi_1)_{L^2}+(\delta_{\phi_2} G_c(\vec{\phi})-\delta_{\phi_2}
G_e(\vec{\psi}),\phi_2-\psi_2)_{L^2}.
     \end{align*}
\end{proof}
	\section{Numerical scheme}
	\label{sec:numerical scheme}
	\setcounter{equation}{0}

	\subsection{Discretization of two-dimensional space}

In the spatial discretization, the centered difference approximation is applied. Some basic notations have to be recalled. We use the notations and results for some discrete functions and operators from~\cite{wise10, wise09a}. Let $\Omega = (0,L_x)\times(0,L_y)$, and we assume $L_x =L_y =: L > 0$ for simplicity of presentation. Let $N\in\mathbb{N}$ be given, and define the grid spacing $h := \nicefrac{L}{N}$.  We also assume -- but only for simplicity of notation, ultimately -- that the mesh spacing in the $x$ and $y$-directions are the same. The following two uniform, infinite grids with grid spacing $h>0$, are introduced: 
	\[
E := \{ p_{i+\hf} \ |\ i\in {\mathbb{Z}}\}, \quad C := \{ p_i \ |\ i\in {\mathbb{Z}}\},
	\]
where $p_i = p(i) := (i-\hf)\cdot h$. Consider the following 2-D discrete $N^2$-periodic function spaces:
\begin{eqnarray*}
	\begin{aligned}
{\mathcal C}_{\rm per} &:= \left\{\nu: C\times C
\rightarrow {\mathbb{R}}\ \middle| \ \nu_{i,j} = \nu_{i+\alpha N,j+\beta N}, \ \forall \,
i,j,\alpha,\beta\in \mathbb{Z} \right\},
	\\
{\mathcal E}^{\rm x}_{\rm per} &:=\left\{\nu: E\times C\rightarrow {\mathbb{R}}\ \middle| \
\nu_{i+\frac12,j}= \nu_{i+\frac12+\alpha N,j+\beta N}, \ \forall \, i,j,\alpha,\beta\in
\mathbb{Z}\right\} .
	\end{aligned}
	\end{eqnarray*}
Here we are using the identification $\nu_{i,j} = \nu(p_i,p_j)$, \emph{et cetera}. The space ${\mathcal E}^{\rm y}_{\rm per}$ is analogously defined. The function of ${\mathcal C}_{\rm per}$ is called {\emph{cell-centered function}}, and the function of ${\mathcal E}^{\rm x}_{\rm per}$ and ${\mathcal E}^{\rm y}_{\rm per}$,  is called {\emph{edge-centered function}}.  We also define the mean zero space
    \[
\mathring{\mathcal C}_{\rm per}:=\left\{\nu\in {\mathcal C}_{\rm per} \ \middle| \ 0 = \overline{\nu} :=  \frac{h^2}{L^2} \sum_{i,j=1}^N \nu_{i,j} \right\} .
	\]
In addition, $\vec{\mathcal{E}}_{\rm per}$ is defined as $\vec{\mathcal{E}}_{\rm per} := {\mathcal E}^{\rm x}_{\rm per}\times {\mathcal E}^{\rm y}_{\rm per}$. We now introduce the difference and average operators on the spaces:
	\begin{eqnarray*}
&& A_x \nu_{i+\hf,j} := \frac{1}{2}\left(\nu_{i+1,j} + \nu_{i,j} \right), \quad D_x
\nu_{i+\hf,j} :=
 \frac{1}{h}\left(\nu_{i+1,j} - \nu_{i,j} \right),\\
&& A_y \nu_{i,j+\hf} := \frac{1}{2}\left(\nu_{i,j+1} + \nu_{i,j} \right), \quad D_y
\nu_{i,j+\hf} := \frac{1}{h}\left(\nu_{i,j+1} - \nu_{i,j} \right) ,
	\end{eqnarray*}
with $A_x,\, D_x: {\mathcal C}_{\rm per}\rightarrow{\mathcal E}_{\rm per}^{\rm x}$, $A_y,\, D_y: {\mathcal C}_{\rm per}\rightarrow{\mathcal E}_{\rm per}^{\rm y}$. Likewise,
\begin{eqnarray*}
&& a_x \nu_{i, j} := \frac{1}{2}\left(\nu_{i+\hf, j} + \nu_{i-\hf, j} \right),	 \quad d_x
\nu_{i,
 j} := \frac{1}{h}\left(\nu_{i+\hf, j} - \nu_{i-\hf, j} \right),\\
&& a_y \nu_{i,j} := \frac{1}{2}\left(\nu_{i,j+\hf} + \nu_{i,j-\hf} \right),	 \quad d_y
\nu_{i,j} := \frac{1}{h}\left(\nu_{i,j+\hf} - \nu_{i,j-\hf} \right),
	\end{eqnarray*}
with $a_x,\, d_x : {\mathcal E}_{\rm per}^{\rm x}\rightarrow{\mathcal C}_{\rm per}$, $a_y,\, d_y : {\mathcal E}_{\rm per}^{\rm y}\rightarrow{\mathcal C}_{\rm per}$. The discrete gradient operator $\nabh:{\mathcal C}_{\rm per}\rightarrow \vec{\mathcal{E}}_{\rm per}$ is given by
    \[
\nabh\nu_{i,j} =\left( D_x\nu_{i+\hf, j},  D_y\nu_{i, j+\hf}\right) ,
	\]
and the discrete divergence $\nabh\cdot :\vec{\mathcal{E}}_{\rm per} \rightarrow {\mathcal C}_{\rm per}$ is defined via
	\[
\nabh\cdot\vec{f}_{i,j} = d_x f^x_{i,j}	+ d_y f^y_{i,j},
	\]
where $\vec{f} = (f^x,f^y)\in \vec{\mathcal{E}}_{\rm per}$. The standard 2-D discrete Laplacian, $\Delta_h : {\mathcal C}_{\rm per}\rightarrow{\mathcal C}_{\rm per}$, becomes
    \begin{align*}
\Delta_h \nu_{i,j} := &  d_x(D_x \nu)_{i,j} + d_y(D_y \nu)_{i,j}
	\\
= & \ \frac{1}{h^2}\left( \nu_{i+1,j}+\nu_{i-1,j}+\nu_{i,j+1}+\nu_{i,j-1} -
4\nu_{i,j}\right).
	\end{align*}
More generally, if $\mathcal{D}$ is a periodic \emph{scalar} function that is defined at all of the edge center points and $\vec{f}\in\vec{\mathcal{E}}_{\rm per}$, then $\mathcal{D}\vec{f}\in\vec{\mathcal{E}}_{\rm per}$, assuming point-wise multiplication, and we may define
	\[
\nabla_h\cdot \big(\mathcal{D} \vec{f} \big)_{i,j} = d_x\left(\mathcal{D}f^x\right)_{i,j}  +
 d_y\left(\mathcal{D}f^y\right)_{i,j}  .
	\]
Specifically, if $\nu\in \mathcal{C}_{\rm per}$, then $\nabla_h \cdot\left(\mathcal{D}
\nabla_h  \ \
\right):\mathcal{C}_{\rm per} \rightarrow \mathcal{C}_{\rm per}$ is defined point-wise via
	\[
\nabla_h\cdot \big(\mathcal{D} \nabla_h \nu \big)_{i,j} =
d_x\left(\mathcal{D}D_x\nu\right)_{i,j}  +
 d_y\left(\mathcal{D} D_y\nu\right)_{i,j} .
	\]
Now we are ready to define the following grid inner products:
	\begin{equation*}
	\begin{aligned}
\ciptwo{\nu}{\xi} &:= h^2\sum_{i,j=1}^N  \nu_{i,j}\, \xi_{i,j},\, \nu,\, \xi\in {\mathcal
C}_{\rm per},\,
&\eipx{\nu}{\xi} := \ciptwo{a_x(\nu\xi)}{1} ,\, \nu,\, \xi\in{\mathcal E}^{\rm x}_{\rm per},
\\
\eipy{\nu}{\xi} &:= \ciptwo{a_y(\nu\xi)}{1} ,\, \nu,\, \xi\in{\mathcal E}^{\rm y}_{\rm per},
	\end{aligned}
	\end{equation*}
	\[
\eipvec{\vec{f}_1}{\vec{f}_2} : = \eipx{f_1^x}{f_2^x}	+ \eipy{f_1^y}{f_2^y} , \quad
\vec{f}_i =
 (f_i^x,f_i^y) \in \vec{\mathcal{E}}_{\rm per}, \ i = 1,2.
	\]
In turn, the following norms could be appropriately introduced for cell-centered functions for $\nu\in {\mathcal C}_{\rm per}$: $\nrm{\nu}_p^p := \ciptwo{|\nu|^p}{1}$, for $1\le p< \infty$, and $\nrm{\nu}_\infty := \max_{1\le i,j\le N}\left|\nu_{i,j}\right|$. We also define norms of the gradient (for $\nu\in{\mathcal C}_{\rm per}$) as follows: 
	\[
\nrm{ \nabla_h \nu}_2^2 : = \eipvec{\nabh \nu }{ \nabh \nu } = \eipx{D_x\nu}{D_x\nu} +
 \eipy{D_y\nu}{D_y\nu},
	\]
and, more generally, 
	\begin{equation*}
\nrm{\nabla_h \nu}_p := \left( \eipx{|D_x\nu|^p}{1} + \eipy{|D_y\nu|^p}{1} \right)^{\frac1p}
, \quad 1 \le p < \infty . 
	\end{equation*}
Higher order norms can be similarly formulated. For example,
	\[
\nrm{\nu}_{H_h^1}^2 : =  \nrm{\nu}_2^2+ \nrm{ \nabla_h \nu}_2^2, \quad \nrm{\nu}_{H_h^2}^2 :
=
\nrm{\nu}_{H_h^1}^2  + \nrm{ \Delta_h \nu}_2^2.
	\]
\begin{lemma}
	\label{lemma1}
Let $\mathcal{D}$ be an arbitrary periodic, scalar function defined on all of the edge-center points. For any $\psi, \nu \in {\mathcal C}_{\rm per}$ and any $\vec{f}\in\vec{\mathcal{E}}_{\rm per}$, the following summation by parts formulas are valid:
	\begin{equation}
\ciptwo{\psi}{\nabla_h\cdot\vec{f}} = - \eipvec{\nabla_h \psi}{ \vec{f}}, \quad
\ciptwo{\psi}{\nabla_h\cdot \left(\mathcal{D}\nabla_h\nu\right)} = - \eipvec{\nabla_h \psi
}{
\mathcal{D}\nabla_h\nu} .
\label{lemma 1-0}
	\end{equation}
	\end{lemma}
To facilitate the analysis below, we need to introduce a discrete analogue of the space $H_{per}^{-1}\left(\Omega\right)$, as outlined in~\cite{wang11a}. Suppose that $\mathcal{D}$ is a positive, periodic scalar function defined at edge-center points. For any $\phi\in{\mathcal C}_{\rm per}$, there exists a unique $\psi\in\mathring{\mathcal C}_{\rm per}$ that solves
\begin{eqnarray}
\mathcal{L}_{\mathcal{D}}(\psi):= - \nabla_h \cdot\left(\mathcal{D}\nabla_h \psi\right) =
\phi - \overline{\phi} ,
	\end{eqnarray}
where $\overline{\phi} := |\Omega|^{-1}\ciptwo{\phi}{1}$. We equip this space with a bilinear form: for any $\phi_1,\, \phi_2\in \mathring{\mathcal C}_{\rm per}$, define
\begin{equation}
\cipgen{ \phi_1 }{ \phi_2 }{\mathcal{L}_{\mathcal{D}}^{-1}} := \eipvec{\mathcal{D}\nabla_h \psi_1 }{\nabla_h \psi_2 },
	\end{equation}
where $\psi_i\in\mathring{\mathcal C}_{\rm per}$ is the unique solution to
	\begin{equation}
\mathcal{L}_{\mathcal{D}}(\psi_i):= - \nabla_h \cdot\left(\mathcal{D}\nabla_h \psi_i\right)  =
\phi_i, \quad i = 1, 2.
	\end{equation}
The following identity~\cite{wang11a} is easy to prove via summation-by-parts:
	\begin{equation}
\cipgen{\phi_1 }{ \phi_2 }{\mathcal{L}_{\mathcal{D}}^{-1}} = \ciptwo{\phi_1}{
\mathcal{L}_{\mathcal{D}}^{-1} (\phi_2) } = \ciptwo{ \mathcal{L}_{\mathcal{D}}^{-1} (\phi_1)
}{\phi_2 },
	\end{equation}
and since $\mathcal{L}_{\mathcal{D}}$ is symmetric positive definite, $\cipgen{ \ \cdot \ }{\ \cdot \ }{\mathcal{L}_{\mathcal{D}}^{-1}}$ is an inner product on $\mathring{\mathcal C}_{\rm per}$~\cite{wang11a}. When $\mathcal{D}\equiv 1$, we drop the subscript and write $\mathcal{L}_{1} = \mathcal{L} = -\Delta_h$, and introduce the notation $\cipgen{ \ \cdot \ }{\ \cdot \ }{\mathcal{L}_{\mathcal{D}}^{-1}} =: \cipgen{ \ \cdot \ }{\ \cdot \ }{-1,h}$. In the general setting, the norm associated to this inner product is denoted $\nrm{\phi}_{\mathcal{L}_{\mathcal{D}}^{-1}} := \sqrt{\cipgen{\phi }{ \phi }{\mathcal{L}_{\mathcal{D}}^{-1}}}$, for all $\phi \in \mathring{\mathcal C}_{\rm per}$, but, if $\mathcal{D}\equiv 1$, we write $\nrm{\, \cdot \, }_{\mathcal{L}_{\mathcal{D}}^{-1}} =: \nrm{\, \cdot \, }_{-1,h}$.

\subsection{A convex-concave decomposition of the discrete energy}

Let us define
	\[
\vec{\mathcal{C}}_{\rm per}^{\mathcal{G}} := \left\{(\phi_1,\phi_2) \in \mathcal{C}_{\rm per}\times \mathcal{C}_{\rm per} \ \middle| \ ({\phi_1}_{i,j},{\phi_2}_{i,j}) \in\mathcal{G} , \quad i,j\in\mathbb{Z}  \right\}, 
	\]
which corresponds to the pairs of periodic grid functions whose point values are in the Gibbs Triangle, $\mathcal{G}$. Define $\kappa(\phi): = \frac{1}{36\phi}$. The discrete energy $G_h(\phi_1,\phi_2):\vec{\mathcal{C}}_{\rm per}^{\mathcal{G}}\rightarrow \mathbb{R}$ is introduced as
	\begin{align}
G_h(\phi_1,\phi_2) & =\ciptwo{ S(\phi_1,\phi_2) + H(\phi_1,\phi_2) }{1}
	\nonumber
	\\
& \quad +\ciptwo{a_x(\kappa(A_x\phi_1)(D_x\phi_1)^2)+a_y(\kappa(A_y\phi_1)(D_y\phi_1)^2)}{\varepsilon_1^2}
	\nonumber
	\\
&\quad +\ciptwo{a_x(\kappa(A_x\phi_2)(D_x\phi_2)^2)+a_y(\kappa(A_y\phi_2)(D_y\phi_2)^2)}{\varepsilon_2^2}
	\nonumber
	\\
& \quad +\ciptwo{a_x(\kappa(A_x (1-\phi_1-\phi_2))(D_x(1-\phi_1-\phi_2))^2)}{\varepsilon_3^2}
	\nonumber
	\\
& \quad +\ciptwo{a_y(\kappa(A_y(1-\phi_1-\phi_2))(D_y(1-\phi_1-\phi_2))^2)}{\varepsilon_3^2}.
	\label{Full-discrete-energy}
	\end{align}

	\begin{lemma}[Existence of a convex-concave decomposition]
	\label{Full-discrete-energy-splitting}
Suppose $(\phi_1,\phi_2) \in \vec{\mathcal{C}}_{\rm per}^{\mathcal{G}}$. The functions
	\begin{align}
G_{h,c}(\phi_1,\phi_2) &: =\ciptwo{S(\phi_1,\phi_2) }{1}
	\label{Discrete-energy-c}
	\\
&  \quad + \ciptwo{a_x(\kappa(A_x\phi_1)(D_x\phi_1)^2)+a_y(\kappa(A_y\phi_1)(D_y\phi_1)^2)}{\varepsilon_1^2}
	\nonumber
	\\
& \quad + \ciptwo{a_x(\kappa(A_x\phi_2)(D_x\phi_2)^2)+a_y(\kappa(A_y\phi_2)(D_y\phi_2)^2)}{\varepsilon_2^2}
	\nonumber
	\\
& \quad +\ciptwo{a_x(\kappa(A_x(1-\phi_1-\phi_2))(D_x(1-\phi_1-\phi_2))^2)}{\varepsilon_3^2}
	\nonumber
	\\
& \quad +\ciptwo{a_y(\kappa(A_y(1-\phi_1-\phi_2))(D_y(1-\phi_1-\phi_2))^2)}{\varepsilon_3^2},	
	\nonumber
	\\
G_{h,e}(\phi_1,\phi_2) &:= -\ciptwo{ H(\phi_1,\phi_2) }{1} ,
	\label{Discrete-energy-e}
	\end{align}
are convex. Therefore, $G_h(\phi_1,\phi_2)=G_{h,c}(\phi_1,\phi_2)-G_{h,e}(\phi_1,\phi_2)$ is a convex-concave decomposition of the discrete energy.
	\end{lemma}

	\begin{proof}
We look at the detailed expansions of $G_{h,c}(\phi_1,\phi_2)$ and $G_{h,e}(\phi_1,\phi_2)$:
	\begin{align*}
G_{h,c}(\phi_1,\phi_2) & = h^2 \sumij \Bigl(
 S({\phi_1}_{i,j},{\phi_2}_{i,j}) \Bigr.
 	\\
& \quad \Bigl.
+\varepsilon_1^2 T_3({\phi_1}_{i+1,j},{\phi_1}_{i,j},D_x{\phi_1}_{i+\hf,j})
+\varepsilon_1^2 T_3({\phi_1}_{i,j},{\phi_1}_{i-1,j},D_x{\phi_1}_{i-\hf,j})
\Bigr.
	\\
& \quad \Bigl.
+\varepsilon_1^2 T_3({\phi_1}_{i,j+1},{\phi_1}_{i,j},D_y{\phi_1}_{i,j+\hf})
+\varepsilon_1^2 T_3({\phi_1}_{i,j},{\phi_1}_{i,j-1},D_y{\phi_1}_{i,j-\hf})
\Bigr.
	\\
& \quad \Bigl.
+\varepsilon_2^2 T_3({\phi_2}_{i+1,j},{\phi_2}_{i,j},D_x{\phi_2}_{i+\hf,j})
+\varepsilon_2^2 T_3({\phi_2}_{i,j},{\phi_2}_{i-1,j},D_x{\phi_2}_{i-\hf,j})
\Bigr.
	\\
& \quad \Bigl.
+\varepsilon_2^2 T_3({\phi_2}_{i,j+1},{\phi_2}_{i,j},D_y{\phi_2}_{i,j+\hf})
+\varepsilon_2^2 T_3({\phi_2}_{i,j},{\phi_2}_{i,j-1},D_y{\phi_2}_{i,j-\hf})
\Bigr.
	\\
& \quad \Bigl.
+\varepsilon_3^2 T_3({(1-\phi_1-\phi_2)}_{i+1,j},{(1-\phi_1-\phi_2)}_{i,j},D_x{(1-\phi_1-\phi_2)}_{i+\hf,j})
\Bigr.
	\\
& \quad \Bigl.
+\varepsilon_3^2 T_3({(1-\phi_1-\phi_2)}_{i,j},{(1-\phi_1-\phi_2)}_{i-1,j},D_x{(1-\phi_1-\phi_2)}_{i-\hf,j})
\Bigr.
	\\
& \quad \Bigl.
+\varepsilon_3^2 T_3({(1-\phi_1-\phi_2)}_{i,j+1},{(1-\phi_1-\phi_2)}_{i,j},D_y{(1-\phi_1-\phi_2)}_{i,j+\hf})
\Bigr.
	\\
& \quad \Bigl.
+\varepsilon_3^2 T_3({(1-\phi_1-\phi_2)}_{i,j},{(1-\phi_1-\phi_2)}_{i,j-1},D_y{(1-\phi_1-\phi_2)}_{i,j-\hf})
\Bigr),
	\\
G_{h,e}(\phi_1,\phi_2) & = -h^2 \sumij H({\phi_1}_{i,j},{\phi_2}_{i,j}).
	\end{align*}
It's clear that $G_{h,c}$ and $G_{h,e}$ are linear combination of certain convex functions;  see the analysis in \Cref{convex splitting}. Therefore, they are both convex.
	\end{proof}

	\begin{proposition}
Suppose $(\phi_1,\phi_2) \in \vec{\mathcal{C}}_{\rm per}^{\mathcal{G}}$. The variational derivatives of $G_{h,c}$ and $G_{h,e}$ with respect to $\phi_1$ and $\phi_2$ are grid functions satisfying	
	\begin{align}
 \delta_{\phi_i}G_{h,c}(\phi_1,\phi_2) & =  \frac{\partial}{\partial \phi_i} S(\phi_1,\phi_2)
	\label{Derivative-energy-c}
 	\\
& \quad +\varepsilon_i^2 a_x(\kappa'(A_x\phi_i)(D_x\phi_i)^2)-2\varepsilon_i^2 d_x(\kappa(A_x\phi_i) D_x\phi_i)
 	\nonumber
 	\\
& \quad + \varepsilon_i^2 a_y(\kappa'(A_y\phi_i)(D_y\phi_i)^2)-2\varepsilon_i^2 d_y(\kappa(A_y\phi_i) D_y\phi_i)
	\nonumber
	\\
& \quad - \varepsilon_3^2 a_x( \kappa'(A_x(1-\phi_1-\phi_2))(D_x(1-\phi_1-\phi_2))^2)
	\nonumber
	\\
& \quad + 2\varepsilon_3^2 d_x(\kappa(A_x(1-\phi_1-\phi_2)) D_x(1-\phi_1-\phi_2) )
	\nonumber
	\\
& \quad - \varepsilon_3^2 a_y( \kappa'(A_y(1-\phi_1-\phi_2))(D_y(1-\phi_1-\phi_2))^2)
	\nonumber
	\\
& \quad  + 2\varepsilon_3^2 d_y(\kappa(A_y(1-\phi_1-\phi_2)) D_y(1-\phi_1-\phi_2) ),
	\nonumber
 	\\
\delta_{\phi_i}G_{h,e}(\phi_1,\phi_2) & = - \frac{\partial}{\partial \phi_i} H(\phi_1,\phi_2),
	\label{Derivative-energy-e}
	\end{align}
for $i = 1,2$.
	\end{proposition}

	\begin{proof}
Fix $(\phi_1,\phi_2) \in \vec{\mathcal{C}}_{\rm per}^{\mathcal{G}}$ and let $\psi_1\in\mathcal{C}_{\rm per}$. Define the function of one variable
	\[
J_{1,c}(\lambda) = G_{h,c}(\phi_1+\lambda \psi_1,\phi_2),
	\]
for all $\lambda\in \mathbb{R}$ sufficiently small that $(\phi_1+\lambda \psi_1,\phi_2)\in \vec{\mathcal{C}}_{\rm per}^{\mathcal{G}}$. The function $J_{1,c}(\lambda)$ is continuous and differentiable. By definition, the variational derivative satisfies
	\[
 J'_{1,c}(0) = \ciptwo{\delta_{\phi_1}G_{h,c}(\phi_1,\phi_2)}{\psi_1}.	
	\]
Since the operators $a_x,A_x,D_x,a_y,A_y$ and $D_y$ are all linear, the following derivation is available
	\begin{align*}
 J'_{1,c}(0) & =\ciptwo{\frac{\partial}{\partial \phi_1} S(\phi_1,\phi_2)\psi_1 }{1}
 	\\
& \quad  + \varepsilon_1^2 \eipx{\kappa'(A_x\phi_1)A_x\psi_1(D_x\phi_1)^2+2\kappa(A_x\phi_1) D_x\phi_1 D_x\psi_1}{A_x 1}
	\\
& \quad  + \varepsilon_1^2 \eipy{\kappa'(A_y\phi_1)A_y\psi_1(D_y\phi_1)^2+2\kappa(A_y\phi_1) D_y\phi_1 D_y\psi_1}{A_y 1}
	\\
& \quad + \varepsilon_3^2 [ -\kappa'(A_x(1-\phi_1-\phi_2)) A_x\psi_1(D_x(1-\phi_1-\phi_2))^2
	\\
& \quad - 2\kappa(A_x(1-\phi_1-\phi_2)) D_x(1-\phi_1-\phi_2) D_x\psi_1 ,  A_x 1 ]_x
 	\\
& \quad + \varepsilon_3^2 [ -\kappa'(A_y(1-\phi_1-\phi_2))A_y\psi_1(D_y(1-\phi_1-\phi_2))^2
 	\\
& \quad - 2 \kappa (A_y(1-\phi_1-\phi_2)) D_y(1-\phi_1-\phi_2) D_y\psi_1 , A_y 1]_y
 	\\
& = \ciptwo{\frac{\partial}{\partial \phi_1} S(\phi_1,\phi_2) }{\psi_1}
 	\\
& \quad + \varepsilon_1^2 \ciptwo{a_x(\kappa'(A_x\phi_1)(D_x\phi_1)^2)-2d_x(\kappa(A_x\phi_1) D_x\phi_1) }{\psi_1}
 	\\
& \quad + \varepsilon_1^2 \ciptwo{a_y(\kappa'(A_y\phi_1)(D_y\phi_1)^2)-2d_y(\kappa(A_y\phi_1) D_y\phi_1) }{\psi_1}
	\\
& \quad + \varepsilon_3^2 \langle -a_x(\kappa'(A_x(1-\phi_1-\phi_2))(D_x(1-\phi_1-\phi_2))^2)
	\\
& \quad + 2d_x(\kappa(A_x(1-\phi_1-\phi_2)) D_x(1-\phi_1-\phi_2) ) , \psi_1 \rangle_\Omega
	\\
& \quad + \varepsilon_3^2 \langle -a_y(\kappa'(A_y(1-\phi_1-\phi_2))(D_y(1-\phi_1-\phi_2))^2)
	\\
& \quad +2d_y(\kappa(A_y(1-\phi_1-\phi_2)) D_y(1-\phi_1-\phi_2) ) , \psi_1 \rangle_\Omega .
	\end{align*}
Therefore,
	\begin{align*}
\delta_{\phi_1}G_{h,c}(\phi_1,\phi_2) & =  \frac{\partial}{\partial \phi_1} S(\phi_1,\phi_2)
	\\
& \quad  +\varepsilon_1^2 a_x(\kappa'(A_x\phi_1)(D_x\phi_1)^2) -2\varepsilon_1^2 d_x(\kappa(A_x\phi_1) D_x\phi_1)
	\\
& \quad + \varepsilon_1^2 a_y(\kappa'(A_y\phi_1)(D_y\phi_1)^2) -2\varepsilon_1^2 d_y(\kappa(A_y\phi_1) D_y\phi_1)
	\\
& \quad  -\varepsilon_3^2 a_x(\kappa'(A_x(1-\phi_1-\phi_2))(D_x(1-\phi_1-\phi_2))^2)
	\\
& \quad + 2\varepsilon_3^2 d_x(\kappa(A_x(1-\phi_1-\phi_2)) D_x(1-\phi_1-\phi_2) )
	\\
& \quad -\varepsilon_3^2 a_y(\kappa'(A_y(1-\phi_1-\phi_2))(D_y(1-\phi_1-\phi_2))^2)
	\\
& \quad  + 2\varepsilon_3^2 d_y(\kappa(A_y(1-\phi_1-\phi_2)) D_y(1-\phi_1-\phi_2) ).
	\end{align*}
The derivations for
$\delta_{\phi_2}G_{h,c}(\phi_1,\phi_2)$, $\delta_{\phi_1}G_{h,e}(\phi_1,\phi_2)$ and
$\delta_{\phi_2}G_{h,e}(\phi_1,\phi_2)$ are quite similar and are omitted for the sake of brevity.
	\end{proof}

	\begin{lemma}
Suppose that $\vec{\phi},\vec{\psi}\in \vec{\mathcal{C}}_{\rm per}^{\mathcal{G}}$. Consider the canonical convex splitting of the energy $G_h(\vec{\phi})$ in~\eqref{Full-discrete-energy} into $G_h=G_{h,c}-G_{h,e}$ given by~\eqref{Discrete-energy-c} -- \eqref{Discrete-energy-e}. The following inequality is available
	
	\begin{align}
G_h(\vec{\phi})-G_h(\vec{\psi}) &\le \ciptwo{\delta_{\phi_1} G_{h,c}(\vec{\phi})-\delta_{\phi_1} G_{h,e}(\vec{\psi})}{\phi_1-\psi_1}
	\label{Full-energy-inequality}
	\\
& \quad +\ciptwo{\delta_{\phi_2} G_{h,c}(\vec{\phi})-\delta_{\phi_2} G_{h,e}(\vec{\psi})}{\phi_2-\psi_2}.
	\nonumber
	\end{align}
	\end{lemma}

	\begin{proof}
Fix $\vec{\phi}\in \vec{\mathcal{C}}_{\rm per}^{\mathcal{G}}$ and $\vec{\varphi}\in\mathcal{C}_{\rm per} \times \mathcal{C}_{\rm per}$. Let $\mathcal{N}\subset\mathbb{R}$ be a sufficiently small neighborhood of $0$. For all $\lambda\in\mathcal{N}$, we can define the continuous and differentiable function $J_c(\lambda) := G_{h,c}(\vec{\phi}+\lambda \vec{\varphi})$. It is clear that $J_c(\lambda)$ is convex, since $G_{h,c}$ is convex. We have $J_c(\lambda)-J_c(0)\geq J'_c(0) \lambda$, for any $\lambda\in\mathcal{N}$. This implies that
	\[
G_{h,c}(\vec{\phi}+\lambda \vec{\varphi}) - G_{h,c}(\vec{\phi}) \geq
 \ciptwo{\delta_{\phi_1}G_{h,c}(\vec{\phi})}{\lambda\varphi_1} +
 \ciptwo{\delta_{\phi_2}G_{h,c}(\vec{\phi})}{\lambda\varphi_2}.
	\]
We may assume that $\vec{\psi}:=\vec{\phi}+\lambda \vec{\varphi}\in \vec{\mathcal{C}}_{\rm per}^{\mathcal{G}}$ since $\lambda$ is small in magnitude. Then we have
	\[
G_{h,c}(\vec{\psi}) - G_{h,c}(\vec{\phi}) \geq \ciptwo{\delta_{\phi_1}G_{h,c}(\vec{\phi})}{\psi_1-\phi_1} + \ciptwo{\delta_{\phi_2}G_{h,c}(\vec{\phi})}{\psi_2-\phi_2}.
	\]
For $G_{h,e}$, we have a similar inequality:
	\[
G_{h,e}(\vec{\psi}) - G_{h,e}(\vec{\phi}) \geq \ciptwo{\delta_{\phi_1}G_{h,e}(\vec{\phi})}{\psi_1-\phi_1} + \ciptwo{\delta_{\phi_2}G_{h,e}(\vec{\phi})}{\psi_2-\phi_2} .
	\]
Combining these inequalities, we obtain  
	\begin{align*}
G_h(\vec{\phi})-G_h(\vec{\psi}) & = \left(G_{h,c}(\vec{\phi})-G_{h,c}(\vec{\psi})\right)
- \left(G_{h,e}(\vec{\phi})-G_{h,e}(\vec{\psi})\right)
	\\
& \leq \ciptwo{\delta_{\phi_1}G_{h,c}(\vec{\phi})}{\phi_1-\psi_1} + \ciptwo{\delta_{\phi_2}G_{h,c}(\vec{\phi})}{\phi_2-\psi_2}
	\\
& \quad - \ciptwo{\delta_{\phi_1}G_{h,e}(\vec{\psi})}{\phi_1-\psi_1} - \ciptwo{\delta_{\phi_2}G_{h,e}(\vec{\psi})}{\phi_2-\psi_2}
	\\
& = \ciptwo{\delta_{\phi_1} G_{h,c}(\vec{\phi})-\delta_{\phi_1} G_{h,e}(\vec{\psi})}{\phi_1-\psi_1}
	\\
& \quad +\ciptwo{\delta_{\phi_2} G_{h,c}(\vec{\phi})-\delta_{\phi_2} G_{h,e}(\vec{\psi})}{\phi_2-\psi_2}.
	\end{align*}
	\end{proof}

Using the standard approach in the convex splitting, the fully discrete scheme is as follows: for
$n\geq 0$, given $(\phi_1^n, \phi_2^n) \in \vec{\mathcal{C}}_{\rm per}^{\mathcal{G}}$, find $(\phi_1^{n+1},\phi_2^{n+1}) \in \vec{\mathcal{C}}_{\rm per}^{\mathcal{G}}$ such that
     \begin{align}
\frac{\phi_1^{n+1}-\phi_1^n}{\dt} & =  \mathcal{M}_1 \Dh \mu_1^{n+1}
	\label{Full-discrete-1},
	\\
\mu_1^{n+1} &: =\delta_{\phi_1}G_{h,c}(\phi_1^{n+1},\phi_2^{n+1})-
\delta_{\phi_1}G_{h,e}(\phi_1^n,\phi_2^n),
	\\
\frac{\phi_2^{n+1}-\phi_2^n}{\dt} & =  \mathcal{M}_2 \Dh \mu_2^{n+1},
	\label{Full-discrete-2}
	\\
\mu_2^{n+1} &:=\delta_{\phi_2}G_{h,c}(\phi_1^{n+1},\phi_2^{n+1})- \delta_{\phi_2}G_{h,e}(\phi_1^n,\phi_2^n).
     \end{align}
%
	
	\section{Positivity-preserving property and unique solvability}
	\label{sec:positivity-preserving property}
	\setcounter{equation}{0}
The proof of the following lemma can be found in~\cite{chen19b}.
 \begin{lemma}
	\label{MMC-positivity-Lem-0}  \cite{chen19b}.	
Suppose that $\phi_1$, $\phi_2 \in \mathcal{C}_{\rm per}$, with $\ciptwo{\phi_1 - \phi_2}{1} = 0$, that is, $\phi_1 - \phi_2\in \mathring{\mathcal{C}}_{\rm per}$, and assume that $\nrm{\phi_1}_\infty < 1$, $\nrm{\phi_2}_\infty \le M$. Then,   we have the following estimate:
     \[
\nrm{(-\Delta_h)^{-1} (\phi_1 - \phi_2)}_\infty \le C_1 ,
     \]
where $C_1>0$ depends only upon $M$ and $\Omega$. In particular, $C_1$ is independent of the mesh size $h$.
\end{lemma}

The following theorem is the main result of the paper. It guarantees the well-defined nature of the proposed scheme.

	\begin{theorem}
	\label{MMC-positivity}
 Given $(\phi_1^n,\phi_2^n)\in\vec{\mathcal{C}}_{\rm per}^{\mathcal{G}}$, then $(\overline{\phi_1^n}, \overline{\phi_2^n})\in\mathcal{G}$, and there exists a unique
 solution $(\phi_1^{n+1},\phi_2^{n+1})\in\vec{\mathcal{C}}_{\rm per}^{\mathcal{G}}$ to
 \eqref{Full-discrete-1} -- \eqref{Full-discrete-2}, with $\overline{\phi_1^n} =
 \overline{\phi_1^{n+1}}$ and $\overline{\phi_2^n} = \overline{\phi_2^{n+1}}$.
	\end{theorem}

	\begin{proof}
For bookkeeping, we introduce the following notation:
	\[
\delta_{\phi_1}G_{h,c}(\phi_1,\phi_2) = \sum_{\ell = 1}^9 Q_\ell(\phi_1,\phi_2),	
	\]
where
	\begin{align*}
Q_1(\phi_1,\phi_2) & := \frac{\partial}{\partial \phi_1} S(\phi_1,\phi_2),
	\\
Q_2(\phi_1,\phi_2) &:= \varepsilon_1^2 a_x(\kappa'(A_x\phi_1)(D_x\phi_1)^2),
	\\
Q_3(\phi_1,\phi_2) &:= -2\varepsilon_1^2 d_x(\kappa(A_x\phi_1) D_x\phi_1),
	\\
Q_4(\phi_1,\phi_2) & := \varepsilon_1^2 a_y(\kappa'(A_y\phi_1)(D_y\phi_1)^2),
	\\
Q_5(\phi_1,\phi_2) & := -2\varepsilon_1^2 d_y(\kappa(A_y\phi_1) D_y\phi_1) ,
	\\
Q_6(\phi_1,\phi_2) & :=-\varepsilon_3^2 a_x(\kappa'(A_x(1-\phi_1-\phi_2))(D_x(1-\phi_1-\phi_2))^2),
	\\
Q_7(\phi_1,\phi_2) & :=  2\varepsilon_3^2 d_x(\kappa(A_x(1-\phi_1-\phi_2)) D_x(1-\phi_1-\phi_2) ),
	\\
Q_8(\phi_1,\phi_2) & := -\varepsilon_3^2 a_y(\kappa'(A_y(1-\phi_1-\phi_2))(D_y(1-\phi_1-\phi_2))^2) ,
	\\
Q_9(\phi_1,\phi_2) & := 2\varepsilon_3^2 d_y(\kappa(A_y(1-\phi_1-\phi_2)) D_y(1-\phi_1-\phi_2) ).
	\end{align*}
The numerical solution of \eqref{Full-discrete-1} -- \eqref{Full-discrete-2} is a minimizer of the following discrete energy functional:
	\begin{align*}
\mathcal{J}_h^n(\phi_1,\phi_2) & = \frac{1}{2\mathcal{M}_1\dt}\|\phi_1-\phi_1^n\|_{-1,h}^2+\frac{1}{2\mathcal{M}_2\dt} \|\phi_2-\phi_2^n\|_{-1,h}^2+\ciptwo{S(\phi_1,\phi_2) }{1}
	\\
& \quad +\ciptwo{a_x(\kappa(A_x\phi_1)(D_x\phi_1)^2) +a_y(\kappa(A_y\phi_1)(D_y\phi_1)^2)}{\varepsilon_1^2}
	\\
& \quad  + \ciptwo{a_x(\kappa(A_x\phi_2)(D_x\phi_2)^2) +a_y(\kappa(A_y\phi_2)(D_y\phi_2)^2)}{\varepsilon_2^2}
	\\
& \quad  + \langle a_x(\kappa(A_x(1-\phi_1-\phi_2))(D_x(1-\phi_1-\phi_2))^2)
	\\
& \quad +a_y(\kappa(A_y(1-\phi_1-\phi_2))(D_y(1-\phi_1-\phi_2))^2) , \varepsilon_3^2 \rangle_\Omega
	\\
& \quad +\ciptwo{\frac{\partial}{\partial \phi_1} H(\phi_1^n,\phi_2^n)}{\phi_1} +\ciptwo{\frac{\partial}{\partial \phi_2} H(\phi_1^n,\phi_2^n)}{\phi_2} ,
	\end{align*}
over the admissible set
	 \[
A_h := \left\{ (\phi_1,\phi_2) \in \vec{\mathcal{C}}_{\rm per}^{\mathcal{G}} \ \middle| \  \ciptwo{\phi_1}{1} = |\Omega| \overline{\phi_1^0}, \quad \ciptwo{\phi_2}{1}= |\Omega| \overline{\phi_2^0} \right\} \subset \mathbb{R}^{2N^2}.
	 \]
It is clear that $\mathcal{J}_h^n$ is a strictly convex functional.

Now, consider the following closed domain:
	\begin{align*}
A_{h,\delta} & := \Bigl\{  (\phi_1,\phi_2) \in \mathcal{C}_{\rm per}\times \mathcal{C}_{\rm per} \ \Big| \ \phi_1,\phi_2 \ge g(\delta), \delta \le \phi_1+\phi_2 \le 1-\delta, \Bigr.
	\\
& \Bigl. \hspace{1.75in} \ciptwo{\phi_1}{1} = |\Omega| \overline{\phi_1^0}, \quad  \ciptwo{\phi_2}{1}= |\Omega| \overline{\phi_2^0} \Bigr\} \subset \mathbb{R}^{2N^2} ,
	\end{align*}
where $g (\delta) >0$ will be given later. Define the hyperplane
	\begin{equation*}
V := \left\{ (\phi_1,\phi_2) \ \middle| \  \ciptwo{\phi_1}{1} = |\Omega| \overline{\phi_1^0}, \quad  \ciptwo{\phi_2}{1}= |\Omega| \overline{\phi_2^0}\right\} \subset \mathbb{R}^{2N^2}.
	\end{equation*}
Since $A_{h,\delta}$ is a bounded, compact, and convex subset of $V$, there exists (not necessarily unique) a minimizer of $\mathcal{J}_h^n(\phi_1,\phi_2)$ over $A_{h,\delta}$. The key point of the positivity analysis is that, such a minimizer could not occur at a boundary point of $A_{h,\delta}$, if $\delta$ and $g(\delta)$ are sufficiently small.

Assume the minimizer of $\mathcal{J}_h^n(\phi_1,\phi_2)$ over $A_{h,\delta}$ occurs at a boundary point of $A_{h,\delta}$.

Case 1: We suppose the minimizer $(\phi_1^{\star},\phi_2^{\star})\in A_{h,\delta}$, satisfies $(\phi_1^{\star})_{\vec{\alpha_0}}=g(\delta)$, for some grid point $\vec{\alpha_0}:=(i_0,j_0)$. Assume that $\phi_1^{\star}$ reaches its maximum value at the grid point $\vec{\alpha_1}:=(i_1,j_1)$.  It is obvious that $(\phi_1^{\star})_{\vec{\alpha_1}}\geq \overline{\phi_1^{\star}}=\overline{\phi_1^0}$. 

A careful calculation gives the following directional derivative
	\begin{align*}
d_s \mathcal{J}_h^n(\phi_1^\star+s\psi,\phi_2^\star)|_{s=0} & = \frac{1}{\mathcal{M}_1\Delta
t}\ciptwo{(-\Delta)^{-1}\left(\phi_1^\star-\phi_1^n \right)}{\psi}
	\\
& \quad + \ciptwo{\delta_{\phi_1}G_{h,c}(\phi_1^{\star},\phi_2^{\star})}{\psi}+
\ciptwo{\frac{\partial}{\partial \phi_1}H(\phi_1^n,\phi_2^n)}{\psi} ,
	\end{align*}
for any $\psi \in \mathring{\mathcal{C}}_{\rm per }$. Let us pick the direction
    \[
\psi_{i,j} = \delta_{i,i_0}\delta_{j,j_0} - \delta_{i,i_1}\delta_{j,j_1},
	\]
where $\delta_{i,j}$ is the Dirac delta function. Note that $\psi$ is of mean zero. The derivative may be expressed as
	\begin{align}
\frac{1}{h^2}d_s \mathcal{J}_h^n(\phi_1^\star+s\psi,\phi_2^\star)|_{s=0} & =
\frac{1}{\mathcal{M}_1\Delta t}(-\Delta)^{-1}\left(\phi_1^\star-\phi_1^n \right)_{\vec{\alpha_0}}
    \label{MMC-Positive}
    \\
& \quad  -
\frac{1}{\mathcal{M}_1\Delta t}(-\Delta)^{-1}\left(\phi_1^\star-\phi_1^n\right)_{\vec{\alpha_1}}
    \nonumber
	\\
& \quad + \delta_{\phi_1}G_{h,c}(\phi_1^{\star},\phi_2^{\star})_{\vec{\alpha_0}}
- \delta_{\phi_1}G_{h,c}(\phi_1^{\star},\phi_2^{\star})_{\vec{\alpha_1}}
	\nonumber
	\\
& \quad + \left(\frac{\partial}{\partial \phi_1}H(\phi_1^n,\phi_2^n)\right)_{\vec{\alpha_0}}
-\left(\frac{\partial}{\partial \phi_1}H(\phi_1^n,\phi_2^n)\right)_{\vec{\alpha_1}}
	\nonumber
	\\
& = \frac{1}{\mathcal{M}_1\Delta t}(-\Delta)^{-1}\left(\phi_1^\star-\phi_1^n \right)_{\vec{\alpha_0}}
    \nonumber
    \\
& \quad - \frac{1}{\mathcal{M}_1\Delta t}(-\Delta)^{-1}\left(\phi_1^\star-\phi_1^n
\right)_{\vec{\alpha_1}}
	\nonumber
	\\
& \quad +\sumi Q_\ell(\phi_1^{\star},\phi_2^{\star})_{\vec{\alpha}_0}
 -\sumi Q_\ell(\phi_1^{\star},\phi_2^{\star})_{\vec{\alpha}_1}
 	\nonumber
 	\\
& \quad +\left(\frac{\partial}{\partial \phi_1}H(\phi_1^n,\phi_2^n)\right)_{\vec{\alpha_0}}
-\left(\frac{\partial}{\partial \phi_1}H(\phi_1^n,\phi_2^n)\right)_{\vec{\alpha_1}}.
	\nonumber
	\end{align}
For the first and second terms appearing in \eqref{MMC-Positive}, we apply
\Cref{MMC-positivity-Lem-0} and obtain
\begin{equation}
- \frac{2C_1}{\mathcal{M}_1} \le
 \frac{1}{\mathcal{M}_1}(-\Delta)^{-1}\left(\phi_1^\star-\phi_1^n \right)_{\vec{\alpha_0}}
-\frac{1}{\mathcal{M}_1}(-\Delta)^{-1}\left(\phi_1^\star-\phi_1^n \right)_{\vec{\alpha_1}}
\le  \frac{2C_1}{\mathcal{M}_1} .
	\label{MMC-positive-1}
	\end{equation}
For the $Q_1$ terms, the following inequality is available:
    \begin{align}
Q_1(\phi_1^{\star},\phi_2^{\star})_{\vec{\alpha}_0} - Q_1(\phi_1^{\star},\phi_2^{\star})_{\vec{\alpha}_1} & = \frac{\partial}{\partial \phi_1} S(\phi_1^{\star},\phi_2^{\star})_{\vec{\alpha}_0} - \frac{\partial}{\partial \phi_1} S(\phi_1^{\star},\phi_2^{\star})_{\vec{\alpha}_1}
	\label{MMC-positive-2}
	\\
& = \left(\frac{1}{M_0}\ln \frac{\alpha (\phi_1^{\star})}{M_0}-  \ln(1-\phi_1^{\star}-\phi_2^{\star})\right)_{\vec{\alpha}_0}	
	\nonumber
	\\
& \quad -  \left(\frac{1}{M_0}\ln \frac{\alpha (\phi_1^{\star})}{M_0} -
\ln(1-\phi_1^{\star}-\phi_2^{\star})\right)_{\vec{\alpha}_1}
	\nonumber
	\\
& =  \left( \ln
\frac{(\phi_1^{\star})^{\nicefrac{1}{M_0}}}{1-\phi_1^{\star}-\phi_2^{\star}}\right)_{\vec{\alpha}_0} - \left( \ln
\frac{(\phi_1^{\star})^{\nicefrac{1}{M_0}}}{1-\phi_1^{\star}-\phi_2^{\star}}\right)_{\vec{\alpha}_1}
	\nonumber
	\\
& \leq   \ln \frac{(g(\delta))^{\nicefrac{1}{M_0}}}{\delta} -
\ln \frac{(\overline{\phi_1^0})^{\nicefrac{1}{M_0}}}{1-\delta}
	\nonumber
	\\
& \leq \ln \frac{(g(\delta))^{\nicefrac{1}{M_0}}}{\delta} -
\frac{1}{M_0}\ln \overline{\phi_1^0}.
	\nonumber
    \end{align}
Using the logarithm property $\ln(ab) = \ln a + \ln b$, we have eliminated the constant $\frac{1}{M_0}\ln \frac{\alpha}{M_0}$. The next-to-last step comes from the facts that $(\phi_1^{\star})_{\vec{\alpha_0}}=g(\delta)$, $(\phi_1^{\star})_{\vec{\alpha_1}}\geq \overline{\phi_1^0}$ and $\delta \le \phi_1+\phi_2 \le 1-\delta$. The last step comes from the inequality that $\ln (1-\delta) < 0$.

For the $Q_2$ terms, we have
    \begin{align}
Q_2(\phi_1^{\star},\phi_2^{\star})_{\vec{\alpha}_0} - Q_2(\phi_1^{\star},\phi_2^{\star})_{\vec{\alpha}_1} & = \varepsilon_1^2 a_x(\kappa'(A_x\phi_1^{\star})(D_x\phi_1^{\star})^2)_{\vec{\alpha}_0}
	\label{MMC-positive-4}
	\\
& \quad - \varepsilon_1^2 a_x(\kappa'(A_x\phi_1^{\star})(D_x\phi_1^{\star})^2)_{\vec{\alpha}_1}
	\nonumber
	\\
& \leq - \varepsilon_1^2 a_x(\kappa'(A_x\phi_1^{\star})(D_x\phi_1^{\star})^2)_{\vec{\alpha}_1}
	\nonumber
	\\
& \leq \frac{\varepsilon_1^2}{9h^2}.
	\nonumber
	\end{align}
The second step above comes from the fact that
	\[
\varepsilon_1^2 a_x(\kappa'(A_x\phi_1^{\star})(D_x\phi_1^{\star})^2)_{\vec{\alpha}_0}\leq 0,
	\]
since $\kappa'(\phi)= -\frac{1}{36\phi^2}<0$. The last step is based on the definitions of $\kappa'(\phi)$, $a_x$, $A_x$, and $D_x$,  as well as the fact that $|\frac{a-b}{a+b}|<1$, $\forall a>0, b>0$. In details, we observe the following expansion
    \begin{align}
- \varepsilon_1^2 a_x(\kappa'(A_x\phi_1^{\star})(D_x\phi_1^{\star})^2)_{\vec{\alpha}_1} & =
\frac{\varepsilon_1^2}{18h^2}
\left[\frac{(\phi_1^{\star})_{i_1+1,j_1}-(\phi_1^{\star})_{i_1,j_1}}{(\phi_1^{\star})_{i_1+1,j_1}+(\phi_1^{\star})_{i_1,j_1}}\right]^2
	\nonumber
	\\
& \quad  + \frac{\varepsilon_1^2}{18h^2}
\left[\frac{(\phi_1^{\star})_{i_1,j_1}-(\phi_1^{\star})_{i_1-1,j_1}}{(\phi_1^{\star})_{i_1+1,j_1}+(\phi_1^{\star})_{i_1,j_1}}\right]^2
	\nonumber
	\\
& \leq \frac{\varepsilon_1^2}{9h^2} .
	\nonumber
	\end{align}
The $Q_4$ terms can be similarly handled:
	\begin{align}
Q_4(\phi_1^{\star},\phi_2^{\star})_{\vec{\alpha}_0} - Q_4(\phi_1^{\star},\phi_2^{\star})_{\vec{\alpha}_1} & = \varepsilon_1^2 a_y(\kappa'(A_y\phi_1^{\star})(D_y\phi_1^{\star})^2)_{\vec{\alpha}_0}
	\label{MMC-positive-5}
	\\
& \quad - \varepsilon_1^2 a_y(\kappa'(A_y\phi_1^{\star})(D_y\phi_1^{\star})^2)_{\vec{\alpha}_1}
	\nonumber
	\\
& \leq \frac{\varepsilon_1^2}{9h^2} .
	\nonumber
	\end{align}

For the $Q_3$ terms, we see that
	\begin{align}
Q_3(\phi_1^{\star},\phi_2^{\star})_{\vec{\alpha}_0} - Q_3(\phi_1^{\star},\phi_2^{\star})_{\vec{\alpha}_1}& = -2\varepsilon_1^2 d_x(\kappa(A_x\phi_1^{\star}) D_x\phi_1^{\star})_{\vec{\alpha}_0}
	\label{MMC-positive-6}
	\\
& \quad + 2\varepsilon_1^2 d_x(\kappa(A_x\phi_1^{\star}) D_x\phi_1^{\star})_{\vec{\alpha}_1}
	\nonumber
	\\
& \leq 0 ,
	\nonumber
	\end{align}
in which the last step comes from the fact that $(D_x\phi_1^{\star})_{i_0-\nicefrac{1}{2},j_0} \leq 0, (D_x\phi_1^{\star})_{i_0+\nicefrac{1}{2},j_0} \geq 0, (D_x\phi_1^{\star})_{i_1-\nicefrac{1}{2},j_1} \geq 0$, and $(D_x\phi_1^{\star})_{i_1+\nicefrac{1}{2},j_1} \leq 0$. 

A bound for the $Q_5$ terms could be similarly derived:
	\begin{align}
Q_5(\phi_1^{\star},\phi_2^{\star})_{\vec{\alpha}_0} - Q_5(\phi_1^{\star},\phi_2^{\star})_{\vec{\alpha}_1} & = -2\varepsilon_1^2 d_y(\kappa(A_y\phi_1^{\star}) D_y\phi_1^{\star})_{\vec{\alpha}_0}
	\label{MMC-positive-7}
	\\
& \quad + 2\varepsilon_1^2 d_y(\kappa(A_y\phi_1^{\star}) D_y\phi_1^{\star})_{\vec{\alpha}_1}
	\nonumber
	\\
& \leq 0 .
	\nonumber
	\end{align}

Use a technique similar to that used for $Q_2$, the $Q_6$ terms could be controlled as follows:
	\begin{align}
Q_6(\phi_1^{\star},\phi_2^{\star})_{\vec{\alpha}_0} - Q_6(\phi_1^{\star},\phi_2^{\star})_{\vec{\alpha}_1} & = -\varepsilon_3^2 a_x\left(\kappa'(A_x(1-\phi_1^{\star}-\phi_2^{\star}))(D_x(1-\phi_1^{\star}-\phi_2^{\star}))^2\right)_{\vec{\alpha}_0}
	\label{MMC-positive-8}
	\\
& \quad + \varepsilon_3^2 a_x\left(\kappa'(A_x(1-\phi_1^{\star}-\phi_2^{\star}))(D_x(1-\phi_1^{\star}-\phi_2^{\star}))^2\right)_{\vec{\alpha}_1}
	\nonumber
	\\
& \leq -\varepsilon_3^2
a_x\left(\kappa'(A_x(1-\phi_1^{\star}-\phi_2^{\star}))(D_x(1-\phi_1^{\star}-\phi_2^{\star}))^2\right)_{\vec{\alpha}_0}
	\nonumber
	\\
& \leq  \frac{\varepsilon_3^2 }{9h^2} .
	\nonumber
	\end{align}
A similar inequality could be derived for the $Q_8$ terms:
    \begin{align}
Q_8(\phi_1^{\star},\phi_2^{\star})_{\vec{\alpha}_0} - Q_8(\phi_1^{\star},\phi_2^{\star})_{\vec{\alpha}_1} & = -\varepsilon_3^2
a_y\left(\kappa'(A_y(1-\phi_1^{\star}-\phi_2^{\star}))(D_y(1-\phi_1^{\star}-\phi_2^{\star}))^2\right)_{\vec{\alpha}_0}
    \label{MMC-positive-9}
	\\
& \quad + \varepsilon_3^2
a_y\left(\kappa'(A_y(1-\phi_1^{\star}-\phi_2^{\star}))(D_y(1-\phi_1^{\star}-\phi_2^{\star}))^2\right)_{\vec{\alpha}_1}
    \nonumber
    \\
& \leq -\varepsilon_3^2
a_y\left(\kappa'(A_y(1-\phi_1^{\star}-\phi_2^{\star}))(D_y(1-\phi_1^{\star}-\phi_2^{\star}))^2\right)_{\vec{\alpha}_0}
    \nonumber
    \\
& \leq  \frac{\varepsilon_3^2}{9h^2} .
	\nonumber
\end{align}

For the $Q_7$ terms, we have
    \begin{align}
Q_7(\phi_1^{\star},\phi_2^{\star})_{\vec{\alpha}_0} -
Q_7(\phi_1^{\star},\phi_2^{\star})_{\vec{\alpha}_1} & = 2 \varepsilon_3^2 d_x(\kappa(A_x(1-\phi_1^{\star}-\phi_2^{\star})) D_x(1-\phi_1^{\star}-\phi_2^{\star}) )_{\vec{\alpha}_0}
    \label{MMC-positive-10}
    \\
& \quad - 2 \varepsilon_3^2
d_x(\kappa(A_x(1-\phi_1^{\star}-\phi_2^{\star})) D_x(1-\phi_1^{\star}-\phi_2^{\star}) )_{\vec{\alpha}_1} 	
    \nonumber
    \\
& = \frac{\varepsilon_3^2}{18h}
\left(\frac{D_x(1-\phi_1^{\star}-\phi_2^{\star})}{A_x(1-\phi_1^{\star}-\phi_2^{\star})}\right)_{i_0+\nicefrac{1}{2},j_0}
    \nonumber
    \\
& \quad -  \frac{\varepsilon_3^2}{18h}
\left(\frac{D_x(1-\phi_1^{\star}-\phi_2^{\star})}
 {A_x(1-\phi_1^{\star}-\phi_2^{\star})}\right)_{i_0-\nicefrac{1}{2},j_0}
    \nonumber
    \\
& \quad - \frac{\varepsilon_3^2}{18h}
\left(\frac{D_x(1-\phi_1^{\star}-\phi_2^{\star})}{A_x(1-\phi_1^{\star}-\phi_2^{\star})}\right)_{i_1+\nicefrac{1}{2},j_1}
    \nonumber
    \\
& \quad + \frac{\varepsilon_3^2}{18h}
\left(\frac{D_x(1-\phi_1^{\star}-\phi_2^{\star})}{A_x(1-\phi_1^{\star}-\phi_2^{\star})}\right)_{i_1-\nicefrac{1}{2},j_1}
	\nonumber
	\\
& \leq \frac{4 \varepsilon_3^2 }{9h^2} .
    \nonumber
	\end{align}
The last step above is based on the definitions of $A_x$ and $D_x$,  as well as the fact that $|\frac{a-b}{a+b}|<1$, $\forall a>0, b>0$.

Similarly, for the $Q_9$ terms, we have
    \begin{align}
Q_9(\phi_1^{\star},\phi_2^{\star})_{\vec{\alpha}_0} - Q_9(\phi_1^{\star},\phi_2^{\star})_{\vec{\alpha}_1} & = 2 \varepsilon_3^2
d_y(\kappa(A_y(1-\phi_1^{\star}-\phi_2^{\star}))D_y(1-\phi_1^{\star}-\phi_2^{\star}))_{\vec{\alpha}_0}
    \label{MMC-positive-11}
    \\
& \quad - 2 \varepsilon_3^2
d_y(\kappa(A_y(1-\phi_1^{\star}-\phi_2^{\star}))D_y(1-\phi_1^{\star}-\phi_2^{\star}))_{\vec{\alpha}_1}
    \nonumber
    \\
& \leq \frac{4 \varepsilon_3^2 }{9 h^2}.
    \nonumber
\end{align}
For the numerical solution $\phi_1^n$ at the previous time step, the a-priori assumption $0<\phi_1^n<1$ indicates that
\begin{equation}
-1 \leq (\phi_1^n)_{\vec{\alpha}_0}-(\phi_1^n)_{\vec{\alpha}_1}\leq 1.
\end{equation}
For the last two terms appearing in \eqref{MMC-Positive}, we see that
    \begin{align}
\left(\frac{\partial}{\partial \phi_1}H(\phi_1^n,\phi_2^n)\right)_{\vec{\alpha_0}} -
\left(\frac{\partial}{\partial \phi_1}H(\phi_1^n,\phi_2^n)\right)_{\vec{\alpha_1}} & = - 2 \chi_{13}
[(\phi_1^n)_{\vec{\alpha}_0}-(\phi_1^n)_{\vec{\alpha}_1}]
    \label{MMC-positive-3}
    \\
& \quad + (\chi_{12}-\chi_{13}-\chi_{23})[(\phi_2^n)_{\vec{\alpha}_0}-(\phi_2^n)_{\vec{\alpha}_1}]
    \nonumber
    \\
& \leq \chi_{12} + 3 \chi_{13}+\chi_{23} .
    \nonumber
    \end{align}
Putting everything together, we have
    \begin{align}
\frac{1}{h^2}d_s \mathcal{J}_h^n(\phi_1^\star+s\psi,\phi_2^\star)|_{s=0} & \leq \ln \frac{(g(\delta))^{\nicefrac{1}{M_0}}}{\delta} - \frac{1}{M_0}\ln\overline{\phi_1^0}+\frac{2C_1}{\mathcal{M}_1\dt}
    \nonumber
    \\
& \quad + \frac{2\varepsilon_1^2 }{9h^2} + \frac{10\varepsilon_3^2 }{9h^2} + \chi_{12} + 3
\chi_{13}+\chi_{23}.
    \nonumber
    \end{align}
The following quantity is introduced: 
	\[
D_0 := - \frac{1}{M_0}\ln \overline{\phi_1^0}+\frac{2C_1}{\mathcal{M}_1\dt} + \frac{2\varepsilon_1^2 }{9h^2}+ \frac{10\varepsilon_3^2 }{9h^2} + \chi_{12}+3 \chi_{13}+\chi_{23}.
	\]
Notice that $D_0$ is a constant for a fixed $\dt, h$, while it becomes singular as $\dt, h \rightarrow 0$. For any fixed $\dt,h$, we could choose $g(\delta)$ small enough so that
	\begin{equation}
\ln \frac{(g(\delta))^{\nicefrac{1}{M_0}}}{\delta} + D_0 < 0.
	\label{MMC-positive-delta}
	\end{equation}
In particular, we can choose
	\[
g(\delta) := (\delta \exp(-D_0-1))^{M_0}.
	\]
This in turn shows that
	\[
\frac{1}{h^2}d_s \mathcal{J}_h^n(\phi_1^\star+s\psi,\phi_2^\star)|_{s=0}  <0,  	\]
 provided that $g(\delta)$ satisfies \eqref{MMC-positive-delta}. But, this contradicts the assumption that $\mathcal{J}_h^n$ has a minimum at $(\phi_1^{\star},\phi_2^{\star})$, since the directional derivative is negative in a direction pointing into $(A_{h,\delta})^{\rm o}$, the interior of $A_{h,\delta}$.

Case 2: Using similar arguments, we are able to prove that, the global minimum of $\mathcal{J}_h^n$ over $A_{h,\delta}$ could not occur on the boundary section where $(\phi_2^{\star})_{\vec{\alpha_0}}=g(\delta)$, if $g(\delta)$ is small enough, for any grid index $\vec{\alpha_0}$.

Case 3: Suppose the minimum point $(\phi_1^{\star},\phi_2^{\star})$  satisfies
	\[
(\phi_1^{\star})_{\vec{\alpha_0}} +
(\phi_2^{\star})_{\vec{\alpha_0}}=1-\delta,
	\]
with $\vec{\alpha_0}:=(i_0,j_0)$. We could choose $\delta \in (0, \nicefrac{1}{3}).$ Without loss of generality, it is assumed that $(\phi_1^{\star})_{\vec{\alpha_0}} \geq \frac{1}{3}$. In addition, we see that 
	\[
\frac{1}{N^2}\sum_{i,j=1}^N (\phi_1+\phi_2)_{i,j}=\overline{\phi_1^0}+\overline{\phi_2^0}.
	\]
There exists one grid point $\vec{\alpha_1}:=(i_1,j_1)$, so that $\phi_1^{\star}+\phi_2^{\star}$ reaches the minimum value at $\vec{\alpha_1}$. Then it is obvious that $(\phi_1^{\star})_{\vec{\alpha_1}}+(\phi_2^{\star})_{\vec{\alpha_1}}\leq
\overline{\phi_1^{\star}}+\overline{\phi_2^{\star}}=\overline{\phi_1^0}+\overline{\phi_2^0}$.
In turn, the following directional derivative could be derived:
\begin{align*}
d_s \mathcal{J}_h^n(\phi_1^\star+s\psi,\phi_2^\star)|_{s=0} =& \frac{1}{\mathcal{M}_1\Delta
t}\ciptwo{(-\Delta)^{-1}\left(\phi_1^\star-\phi_1^n \right)}{\psi}\\
&+ \ciptwo{\delta_{\phi_1}G_{h,c}(\phi_1^{\star},\phi_2^{\star})}{\psi}+
\ciptwo{\frac{\partial}{\partial \phi_1}H(\phi_1^n,\phi_2^n)}{\psi} ,
\end{align*}
for any $\psi \in \mathring{{\cal C}}_{\rm per}$. Setting  the direction as
    \[
\psi_{i,j} = \delta_{i,i_0}\delta_{j,j_0} - \delta_{i,i_1}\delta_{j,j_1},
	\]
then the derivative may be expanded as
    \begin{align}
\frac{1}{h^2}d_s \mathcal{J}_h^n(\phi_1^\star+s\psi,\phi_2^\star)|_{s=0} & =
\frac{1}{\mathcal{M}_1\Delta t}(-\Delta)^{-1}\left(\phi_1^\star-\phi_1^n \right)_{\vec{\alpha_0}}
    \label{MMC-Positive-right}
    \\
& \quad - \frac{1}{\mathcal{M}_1\Delta t}(-\Delta)^{-1}\left(\phi_1^\star-\phi_1^n
\right)_{\vec{\alpha_1}}
    \nonumber
    \\
& \quad + \delta_{\phi_1}G_{h,c}(\phi_1^{\star},\phi_2^{\star})_{\vec{\alpha_0}}
- \delta_{\phi_1}G_{h,c}(\phi_1^{\star},\phi_2^{\star})_{\vec{\alpha_1}}
    \nonumber
    \\
& \quad + \left(\frac{\partial}{\partial \phi_1}H(\phi_1^n,\phi_2^n)\right)_{\vec{\alpha_0}}
- \left(\frac{\partial}{\partial \phi_1}H(\phi_1^n,\phi_2^n)\right)_{\vec{\alpha_1}}
    \nonumber
    \\
& = \frac{1}{\mathcal{M}_1\Delta t}(-\Delta)^{-1}\left(\phi_1^\star-\phi_1^n \right)_{\vec{\alpha_0}}
    \nonumber
    \\
& \quad - \frac{1}{\mathcal{M}_1\Delta t}(-\Delta)^{-1}\left(\phi_1^\star-\phi_1^n
\right)_{\vec{\alpha_1}}
    \nonumber
    \\
& \quad + \sumi Q_\ell(\phi_1^{\star},\phi_2^{\star})_{\vec{\alpha}_0} -
\sumi Q_\ell(\phi_1^{\star},\phi_2^{\star})_{\vec{\alpha}_1}
    \nonumber
    \\
& \quad +\left(\frac{\partial}{\partial \phi_1}H(\phi_1^n,\phi_2^n)\right)_{\vec{\alpha_0}}
-\left(\frac{\partial}{\partial \phi_1}H(\phi_1^n,\phi_2^n)\right)_{\vec{\alpha_1}}.
    \nonumber
\end{align}
For the first and second terms appearing in \eqref{MMC-Positive-right}, we apply  \Cref{MMC-positivity-Lem-0} and obtain
\begin{equation}
- 2 C_1 \le
 (-\Delta)^{-1}\left(\phi_1^\star-\phi_1^n \right)_{\vec{\alpha_0}}
-(-\Delta)^{-1}\left(\phi_1^\star-\phi_1^n \right)_{\vec{\alpha_1}}
\le  2 C_1 ,
	\label{MMC-positive-right-1}
	\end{equation}
    \begin{align}
Q_1(\phi_1^{\star},\phi_2^{\star})_{\vec{\alpha}_0} -
Q_1(\phi_1^{\star},\phi_2^{\star})_{\vec{\alpha}_1} & =
\frac{\partial}{\partial \phi_1} S(\phi_1^{\star},\phi_2^{\star})_{\vec{\alpha}_0}
- \frac{\partial}{\partial \phi_1} S(\phi_1^{\star},\phi_2^{\star})_{\vec{\alpha}_1}
    \label{MMC-positive-right-2}
    \\
& = \left(\frac{1}{M_0}\ln \frac{\alpha (\phi_1^{\star})}{M_0}-
\ln(1-\phi_1^{\star}-\phi_2^{\star})\right)_{\vec{\alpha}_0}
    \nonumber
    \\
& \quad - \left(\frac{1}{M_0}\ln \frac{\alpha (\phi_1^{\star})}{M_0} -
\ln(1-\phi_1^{\star}-\phi_2^{\star})\right)_{\vec{\alpha}_1}
    \nonumber
    \\
& = \left( \ln
\frac{(\phi_1^{\star})^{\nicefrac{1}{M_0}}}{1-\phi_1^{\star}-\phi_2^{\star}}\right)_{\vec{\alpha}_0}
- \left( \ln
\frac{(\phi_1^{\star})^{\nicefrac{1}{M_0}}}{1-\phi_1^{\star}-\phi_2^{\star}}\right)_{\vec{\alpha}_1}
    \nonumber
    \\
& \geq \ln \frac{(\frac{1}{3})^{\nicefrac{1}{M_0}}}{\delta} -
\ln \frac{1}{1-\overline{\phi_1^0}-\overline{\phi_2^0}} .
    \nonumber
    \end{align}
The last step above comes from the facts that $(\phi_1^{\star})_{\vec{\alpha_0}} \geq \frac{1}{3}$,
$(\phi_1^{\star})_{\vec{\alpha_1}}+(\phi_2^{\star})_{\vec{\alpha_1}}\leq
\overline{\phi_1^0}+\overline{\phi_2^0}$, and $(\phi_1^{\star})_{\vec{\alpha_1}} < 1$ .

For the $Q_2$ and $Q_4$ terms, we have
    \begin{align}
Q_2(\phi_1^{\star},\phi_2^{\star})_{\vec{\alpha}_0} -
Q_2(\phi_1^{\star},\phi_2^{\star})_{\vec{\alpha}_1} & = \varepsilon_1^2
a_x(\kappa'(A_x\phi_1^{\star})(D_x\phi_1^{\star})^2)_{\vec{\alpha}_0}
    \label{MMC-positive-right-4}
    \\
& \quad - \varepsilon_1^2
a_x(\kappa'(A_x\phi_1^{\star})(D_x\phi_1^{\star})^2)_{\vec{\alpha}_1}
    \nonumber
    \\
& \geq \varepsilon_1^2 a_x(\kappa'(A_x\phi_1^{\star})(D_x\phi_1^{\star})^2)_{\vec{\alpha}_0}
    \nonumber
    \\
& \geq -\frac{\varepsilon_1^2 }{9h^2} ,
    \nonumber
\end{align}
in which the second step comes from the fact that
$-\varepsilon_1^2 a_x(\kappa'(A_x\phi_1^{\star})(D_x\phi_1^{\star})^2)_{\vec{\alpha}_1}\geq 0$,
since $\kappa'(\phi)= -\frac{1}{36\phi^2}<0$, and the last step is based on the definitions of
$\kappa'(\phi)$, $a_x$, $A_x$, and $D_x$,  as well as the fact that $|\frac{a-b}{a+b}|<1$,
$\forall a>0, b>0$.

For the $Q_4$ terms, similarly, we get 
    \begin{align}
Q_4(\phi_1^{\star},\phi_2^{\star})_{\vec{\alpha}_0} -
Q_4(\phi_1^{\star},\phi_2^{\star})_{\vec{\alpha}_1} & = \varepsilon_1^2
a_y(\kappa'(A_y\phi_1^{\star})(D_y\phi_1^{\star})^2)_{\vec{\alpha}_0}
    \label{MMC-positive-right-5}
    \\
& \quad - \varepsilon_1^2 a_y(\kappa'(A_y\phi_1^{\star})(D_y\phi_1^{\star})^2)_{\vec{\alpha}_1}
    \nonumber
    \\
& \geq -\frac{\varepsilon_1^2 }{9h^2} .
    \nonumber
\end{align}
The $Q_3$ and $Q_5$ terms could be analyzed as follows
    \begin{align}
Q_3(\phi_1^{\star},\phi_2^{\star})_{\vec{\alpha}_0} -
Q_3(\phi_1^{\star},\phi_2^{\star})_{\vec{\alpha}_1} & = -2\varepsilon_1^2
d_x(\kappa(A_x\phi_1^{\star}) D_x\phi_1^{\star})_{\vec{\alpha}_0}
    \label{MMC-positive-right-6}
    \\
& \quad + 2\varepsilon_1^2
d_x(\kappa(A_x\phi_1^{\star}) D_x\phi_1^{\star})_{\vec{\alpha}_1}
    \nonumber
    \\
& \geq -\frac{4\varepsilon_1^2 }{9h^2} ,
    \nonumber
    \\
Q_5(\phi_1^{\star},\phi_2^{\star})_{\vec{\alpha}_0} -
Q_5(\phi_1^{\star},\phi_2^{\star})_{\vec{\alpha}_1} & = -2\varepsilon_1^2
d_y(\kappa(A_y\phi_1^{\star}) D_y\phi_1^{\star})_{\vec{\alpha}_0}
    \label{MMC-positive-right-7}
    \\
& \quad + 2\varepsilon_1^2 d_y(\kappa(A_y\phi_1^{\star}) D_y\phi_1^{\star})_{\vec{\alpha}_1}
    \nonumber
    \\
& \geq -\frac{4\varepsilon_1^2 }{9h^2} .    \nonumber
\end{align}
The estimates for $Q_6$ and $Q_8$ terms are similar:
    \begin{align}
Q_6(\phi_1^{\star},\phi_2^{\star})_{\vec{\alpha}_0} -
Q_6(\phi_1^{\star},\phi_2^{\star})_{\vec{\alpha}_1} & = -\varepsilon_3^2
a_x\left(\kappa'(A_x(1-\phi_1^{\star}-\phi_2^{\star}))(D_x(1-\phi_1^{\star}-\phi_2^{\star}))^2
\right)_{\vec{\alpha}_0}
    \label{MMC-positive-right-8}
    \\
& \quad + \varepsilon_3^2
a_x\left(\kappa'(A_x(1-\phi_1^{\star}-\phi_2^{\star}))(D_x(1-\phi_1^{\star}-\phi_2^{\star}))^2\right)_{\vec{\alpha}_1}
    \nonumber
    \\
& \geq + \varepsilon_3^2
a_x\left(\kappa'(A_x(1-\phi_1^{\star}-\phi_2^{\star}))(D_x(1-\phi_1^{\star}-\phi_2^{\star}))^2\right)_{\vec{\alpha}_1}
    \nonumber
    \\
& \geq - \frac{\varepsilon_3^2 }{9h^2} ,
    \nonumber
    \\
Q_8(\phi_1^{\star},\phi_2^{\star})_{\vec{\alpha}_0} -
Q_8(\phi_1^{\star},\phi_2^{\star})_{\vec{\alpha}_1} & = -\varepsilon_3^2
a_y\left(\kappa'(A_y(1-\phi_1^{\star}-\phi_2^{\star}))(D_y(1-\phi_1^{\star}-\phi_2^{\star}))^2\right)_{\vec{\alpha}_0}
    \label{MMC-positive-right-9}
    \\
& \quad + \varepsilon_3^2
a_y\left(\kappa'(A_y(1-\phi_1^{\star}-\phi_2^{\star}))(D_y(1-\phi_1^{\star}-\phi_2^{\star}))^2\right)_{\vec{\alpha}_1}
    \nonumber
    \\
& \geq + \varepsilon_3^2
a_y\left(\kappa'(A_y(1-\phi_1^{\star}-\phi_2^{\star}))(D_y(1-\phi_1^{\star}-\phi_2^{\star}))^2\right)_{\vec{\alpha}_1}
    \nonumber
    \\
& \geq -\frac{\varepsilon_3^2 }{9h^2} .
    \nonumber
\end{align}
For the $Q_7$ terms, we see that
    \begin{align}
Q_7(\phi_1^{\star},\phi_2^{\star})_{\vec{\alpha}_0} -
Q_7(\phi_1^{\star},\phi_2^{\star})_{\vec{\alpha}_1} & = 2\varepsilon_3^2
d_x(\kappa(A_x(1-\phi_1^{\star}-\phi_2^{\star}))D_x(1-\phi_1^{\star}-\phi_2^{\star}))_{\vec{\alpha}_0}
    \label{MMC-positive-right-10}
    \\
& \quad - 2\varepsilon_3^2 d_x(\kappa(A_x(1-\phi_1^{\star}-\phi_2^{\star})) D_x(1-\phi_1^{\star}-\phi_2^{\star}) )_{\vec{\alpha}_1}
    \nonumber
    \\
& \geq 0 .
    \nonumber
	\end{align}
The last step above comes from the fact that
	\begin{align*}
(D_x (1-\phi_1^{\star}-\phi_2^{\star}))_{i_0-\nicefrac{1}{2},j_0} & \leq 0,
	\\
(D_x(1-\phi_1^{\star}-\phi_2^{\star}))_{i_0+\nicefrac{1}{2},j_0} & \geq 0,
	\\
(D_x(1-\phi_1^{\star}-\phi_2^{\star}))_{i_1-\nicefrac{1}{2},j_1} & \geq 0,
	\\
(D_x(1-\phi_1^{\star}-\phi_2^{\star}))_{i_1+\nicefrac{1}{2},j_1} &\leq 0.
	\end{align*}
Similarly, for the $Q_9$ terms, we see that
    \begin{align}
Q_9(\phi_1^{\star},\phi_2^{\star})_{\vec{\alpha}_0} -
Q_9(\phi_1^{\star},\phi_2^{\star})_{\vec{\alpha}_1} & = + 2\varepsilon_3^2
d_y(\kappa(A_y(1-\phi_1^{\star}-\phi_2^{\star}))D_y(1-\phi_1^{\star}-\phi_2^{\star}))_{\vec{\alpha}_0}
    \label{MMC-positive-right-11}
    \\
& \quad - 2\varepsilon_3^2
d_y(\kappa(A_y(1-\phi_1^{\star}-\phi_2^{\star}))D_y(1-\phi_1^{\star}-\phi_2^{\star}))_{\vec{\alpha}_1}
    \nonumber
    \\
& \geq 0 ,
    \nonumber
\end{align}
For the numerical solution $\phi_1^n$ at the previous time step, similar bounds could be derived for the last two terms appearing in \eqref{MMC-Positive-right}
    \begin{align}
\left(\frac{\partial}{\partial \phi_1}H(\phi_1^n,\phi_2^n)\right)_{\vec{\alpha_0}}
-\left(\frac{\partial}{\partial \phi_1}H(\phi_1^n,\phi_2^n)\right)_{\vec{\alpha_1}} & =
- 2 \chi_{13} [(\phi_1^n)_{\vec{\alpha}_0}-(\phi_1^n)_{\vec{\alpha}_1}]
    \label{MMC-positive-right-3}
    \\
& \quad + (\chi_{12}-\chi_{13}-\chi_{23})[(\phi_2 ^n)_{\vec{\alpha}_0}-(\phi_2
^n)_{\vec{\alpha}_1}]
    \nonumber
    \\
& \geq -\chi_{12}-3 \chi_{13}-\chi_{23} .
    \nonumber
    \end{align}
Putting estimates together, we arrive at 
	\begin{align}
\frac{1}{h^2}d_s \mathcal{J}_h^n(\phi_1^\star+s\psi,\phi_2^\star)|_{s=0} & \geq \ln \frac{(\frac{1}{3})^{\nicefrac{1}{M_0}}}{\delta} -
\ln \frac{1}{1-\overline{\phi_1^0}-\overline{\phi_2^0}}-\frac{2C_1}{\mathcal{M}_1\dt}
    \nonumber
    \\
& \quad - \frac{10\varepsilon_1^2}{9h^2} - \frac{2\varepsilon_3^2}{9h^2} - \chi_{12}-3 \chi_{13}-\chi_{23}.
    \nonumber
	\end{align}
The following quantity is introduced: 
	\[
D_1 := \frac{1}{M_0} \ln3 +\ln
\frac{1}{1-\overline{\phi_1^0}-\overline{\phi_2^0}}+\frac{2C_1}{\mathcal{M}_1\dt} +  \frac{10\varepsilon_1^2}{9h^2}+\frac{2\varepsilon_3^2}{9h^2} + \chi_{12}+3 \chi_{13}+\chi_{23},
	\]
For any fixed $\dt,h$, we could choose $\delta$ small enough so that
	\begin{equation}
-\ln \delta- D_1 > 0,
	\label{MMC-positive-right-delta}
	\end{equation}
in particular, $\delta = \min \{\exp(-D_1-1), \nicefrac{1}{3} \}$. This in turn shows that
\[
\frac{1}{h^2}d_s \mathcal{J}_h^n(\phi_1^\star+s\psi,\phi_2^\star)|_{s=0} > 0,
	\]
provided that $\delta$ satisfies  \eqref{MMC-positive-right-delta}. This contradicts the assumption that $\mathcal{J}_h^n$ has a minimum at
$(\phi_1^{\star},\phi_2^{\star})$.

Case 4: Using similar arguments, we can also prove that, the global minimum of $\mathcal{J}_h^n$ over $A_{h,\delta}$ could not occur on the boundary section where $(\phi_1^{\star})_{\vec{\alpha_0}} + (\phi_2^{\star})_{\vec{\alpha_0}} = \delta$, if $\delta$ is small enough, for any grid index $\vec{\alpha_0}$. The details are left to the interested readers.

Finally, a combination of these four cases reveals that, the global minimizer of $\mathcal{J}_h^n(\phi_1,\phi_2)$ could only possibly occur at interior point of $( A_{h,\delta} )^0 \subset (A_h )^0$. We conclude that there must be a solution $(\phi_1, \phi_2) \in (A_h )^0$ that minimizes $\mathcal{J}_h^n(\phi_1,\phi_2)$ over $A_h$, which is equivalent to the numerical solution of~\eqref{Full-discrete-1} -- \eqref{Full-discrete-2}. The existence of the numerical solution is established.

In addition, since $\mathcal{J}_h^n(\phi_1,\phi_2)$ is a strictly convex function over $A_h$, the uniqueness analysis for this numerical solution is straightforward. The proof of Theorem~\ref{MMC-positivity} is complete.
	\end{proof}

\begin{remark}
For the two-phase MMC model with Flory-Huggins-deGennes free energy density, the energy functional could be represented in terms of a single phase variable, and the positivity-preserving property has been established for the energy stable numerical schemes~\cite{Dong2019a, Dong2020a}. However, a theoretical justification of this property for the ternary MMC system is much more complicated, due to the mixed terms involved in the highly nonlinear and singular surface diffusion part. For example, to overcome the difficulty associated with the coupling between the $\phi_1$ and $\phi_3$ variables in the surface diffusion energy, we have to set different lower and upper bounds for the two variables in the set-up of $A_{h,\delta}$, and a nonlinear scaling (such as~\eqref{MMC-positive-delta}) between $g (\delta)$ and $\delta$ is needed, which turns out to be a crucial step in the nonlinear analysis.
\end{remark}
\section{Unconditional energy stability}\label{sec:unconditional energy
stability}\setcounter{equation}{0}
\begin{theorem}(Energy stability)
The fully discrete scheme~\eqref{Full-discrete-1} -- \eqref{Full-discrete-2} is unconditionally energy stable, i.e., for any time step size $\dt>0$, we have
\begin{equation}
G_h(\phi_1^{n+1},\phi_2^{n+1}) \leq G_h(\phi_1^n,\phi_2^n). \label{Full-energy-stable}
\end{equation}
\end{theorem}

	\begin{proof}
Let $\vec{\phi}=(\phi_1^{n+1},\phi_2^{n+1})$, and $\vec{\psi}=(\phi_1^n,\phi_2^n)$ in the \eqref{Full-energy-inequality}. Applying the fully discrete scheme~\eqref{Full-discrete-1} -- \eqref{Full-discrete-2} and Lemma \ref{lemma1}, we have
     \begin{align*}
& \hspace{-0.5in} G_h(\phi_1^{n+1},\phi_2^{n+1})-G_h(\phi_1^n,\phi_2^n)
 	\\
& \leq \ciptwo{\delta_{\phi_1} G_{h,c}(\phi_1^{n+1},\phi_2^{n+1})-\delta_{\phi_1} G_{h,e}(\phi_1^n,\phi_2^n)}{\phi_1^{n+1}-\phi_1^n}
	\\
& \quad +\ciptwo{\delta_{\phi_2} G_{h,c}(\phi_1^{n+1},\phi_2^{n+1})-\delta_{\phi_2} G_{h,e}(\phi_1^n,\phi_2^n)}{\phi_2^{n+1}-\phi_2^n}
	\\
& =  \ciptwo{\mu_1^{n+1}}{\phi_1^{n+1}-\phi_1^n}+\ciptwo{\mu_2^{n+1}}{\phi_2^{n+1}-\phi_2^n}
	\\
& =\mathcal{M}_1 \ciptwo{\mu_1^{n+1}}{\dt \Dh \mu_1^{n+1}}+\mathcal{M}_2 \ciptwo{\mu_2^{n+1}}{\dt \Dh \mu_2^{n+1}}
	\\
& = -\mathcal{M}_1 \dt\|\nabla_h \mu_1^{n+1}\|_2^2-\mathcal{M}_2 \dt\|\nabla_h \mu_2^{n+1}\|_2^2
	\\
& \leq 0 .
    \end{align*}
	\end{proof}
%
\section{Numerical results}\label{sec:numerical results}\setcounter{equation}{0}
In this section, we present several numerical experiments based on the proposed scheme. The nonlinear Full Approximation Scheme (FAS) multigrid method is used for solving the semi-implicit numerical scheme \eqref{Full-discrete-1} -- \eqref{Full-discrete-2}. The details are similar to earlier works~\cite{baskaran13a, chen19b, Dong2018Convergence, feng2018bsam, guo16, hu09, wise10}, etc. We take the domain as $\Omega = [0,64]^2$, fix the space resolution $N = 256$ and choose the parameters in the model as $M_0 = 0.16, N_0 = 5.12,  \chi_{12} =4, \chi_{13} = 10, \chi_{23} = 1.6$ and $ \mathcal{M}_1 = \mathcal{M}_2 = 1.0$.
\begin{example}\label{example 1}
The initial data is set as
\begin{eqnarray}\label{eqn:init1}
&\phi_1^0(x,y) = 0.1+0.01\cos\big(3\pi x/32\big) \cos\big(3\pi y/32\big),\\
&\phi_2^0(x,y) = 0.5+0.01\cos\big(3\pi x/32\big) \cos\big(3\pi y/32\big).\nonumber
\end{eqnarray}
\end{example}
This example is designed to study the numerical accuracy in time. Since the exact solution is unknown, we treat the numerical solution obtained by $\dt = 1.0\times 10^{-6}$ as the ``exact solution" to calculate the error at the final time. The $\ell^2$ and $\ell^\infty$ errors for $\phi_1$~and $\phi_2$ are displayed in \Cref{tab:cov1}~and \Cref{tab:cov2}, respectively.
\begin{table}
\begin{center}
\begin{tabular}{|c|c|c|c|c|}
\hline $\Delta t$& $8\delta t$& $4\delta t$& $2\delta t$& $\delta t$  \\
\hline $\ell^2$-error-$\phi_1$& $9.5934 \times 10^{-8}$&$ 4.7472 \times 10^{-8}$ &$ 2.3249 \times
10^{-8}$
            &$1.1140 \times 10^{-8}$\\
\hline Rate&-& 1.0150  &   1.0299  & 1.0615\\
\hline $\ell^2$-error-$\phi_2$& $7.0928 \times 10^{-7}$&$ 3.5108 \times 10^{-7}$ &$ 1.7196 \times
10^{-7}$
            &$8.2400 \times 10^{-8}$\\
\hline Rate&-& 1.0146  &   1.0297  & 1.0614
\\
\hline
\end{tabular}
\caption{The $\ell^2$ error and convergence rate for $\phi_1$ and $\phi_2$. The initial data are defined in \eqref{eqn:init1}. The parameters are given by: $T = 0.8$, $\delta t = 1.25\times 10^{-5}$ and $\varepsilon_1=\varepsilon_2=\varepsilon_3=1.0$.} \label{tab:cov1}
\end{center}
\end{table}
\begin{table}
\begin{center}
\begin{tabular}{|c|c|c|c|c|}
\hline $\Delta t$& $8\delta t$& $4\delta t$& $2\delta t$& $\delta t$  \\
\hline $\ell^\infty$-error-$\phi_1$& $1.9507 \times 10^{-7}$&$ 9.6531 \times 10^{-8}$ &$ 4.7275 \times
10^{-8}$
            &$2.2654 \times 10^{-8}$\\
\hline Rate&-& 1.0150  &   1.0299  & 1.0613\\
\hline $\ell^\infty$-error-$\phi_2$& $1.4499 \times 10^{-6}$&$ 7.1765 \times 10^{-7}$ &$ 3.5151 \times
10^{-7}$
            &$1.6844 \times 10^{-7}$\\
\hline Rate&-& 1.0146  &   1.0297  & 1.0614
\\
\hline
\end{tabular}
\caption{The $\ell^\infty$ error and convergence rate for $\phi_1$ and $\phi_2$, with the same initial data and physical parameters as in Table~\ref{tab:cov1}.  }
\label{tab:cov2}
\end{center}
\end{table}
In addition, the energy evolution of the numerical solution with $\dt = 1.0\times 10^{-4}$ is illustrated in~\Cref{fig:coslong_energy}, which indicates a clear energy decay. We also present the error evolution of the total mass of $\phi_1$ and $\phi_2$ in~\cref{fig:cosmass}. In \Cref{fig:coslong}, the snapshot plots of $\phi_1$, $\phi_2$ and $\phi_3$ at a sequence of time instants are displayed, to make a comparison with the existing binary MMC results. Moreover, the maximum values and minimum values of $\phi_1$, $\phi_2$ and $\phi_1+\phi_2$ are presented in \Cref{fig:cosmaxmin} and \Cref{fig:coslong_maxminadd}.
\begin{figure}[!htp]
\begin{center}	
\includegraphics[width=2.5in]{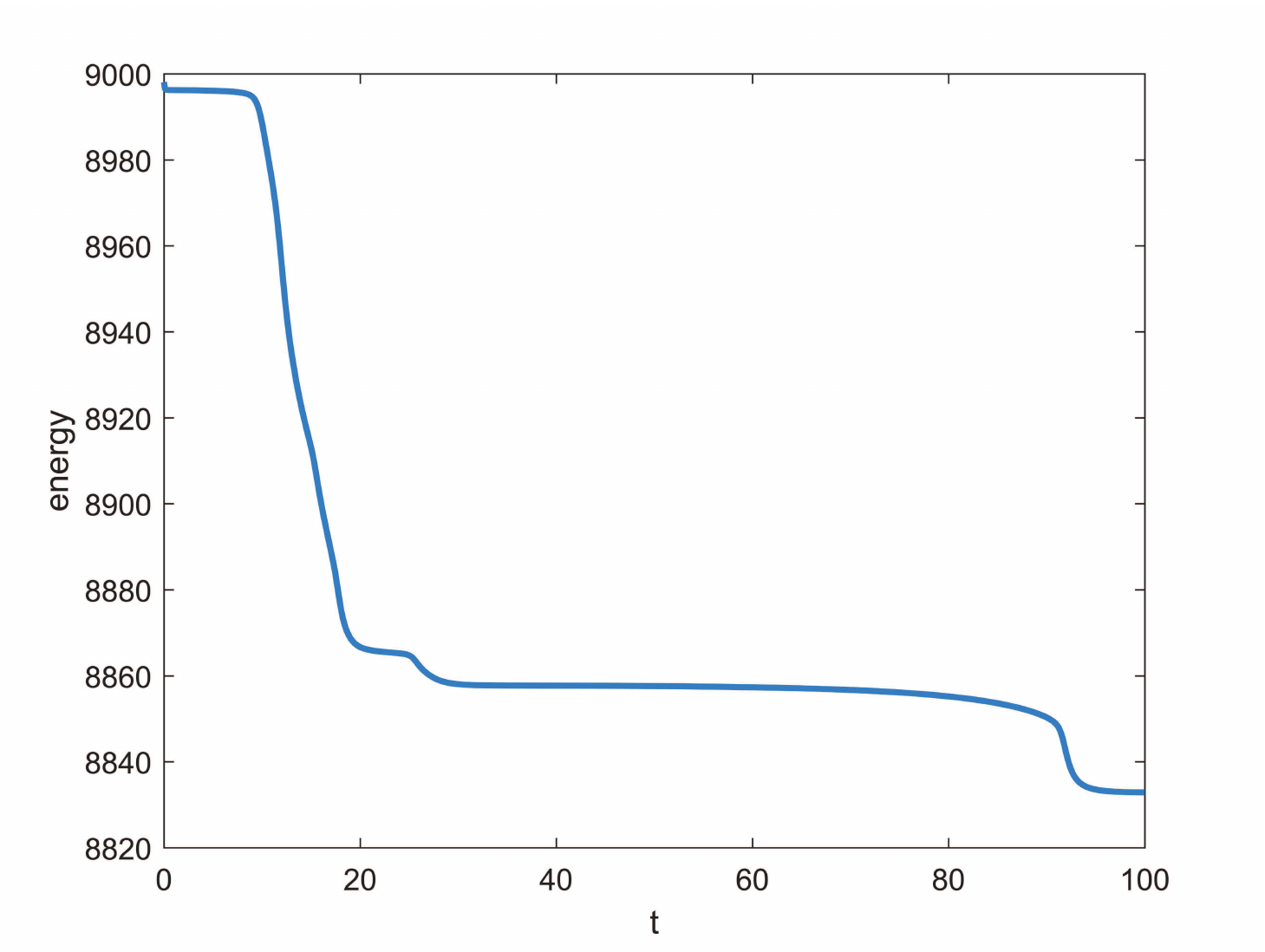}
	\caption{Example \ref{example 1}: Evolution of the energy over time, $\dt = 1.0 \times 10^{-4}$. }\label{fig:coslong_energy}
\end{center}
\end{figure}
\begin{figure}[ht]
	\begin{center}
		\begin{subfigure}{}
			\includegraphics[width=2.2in]{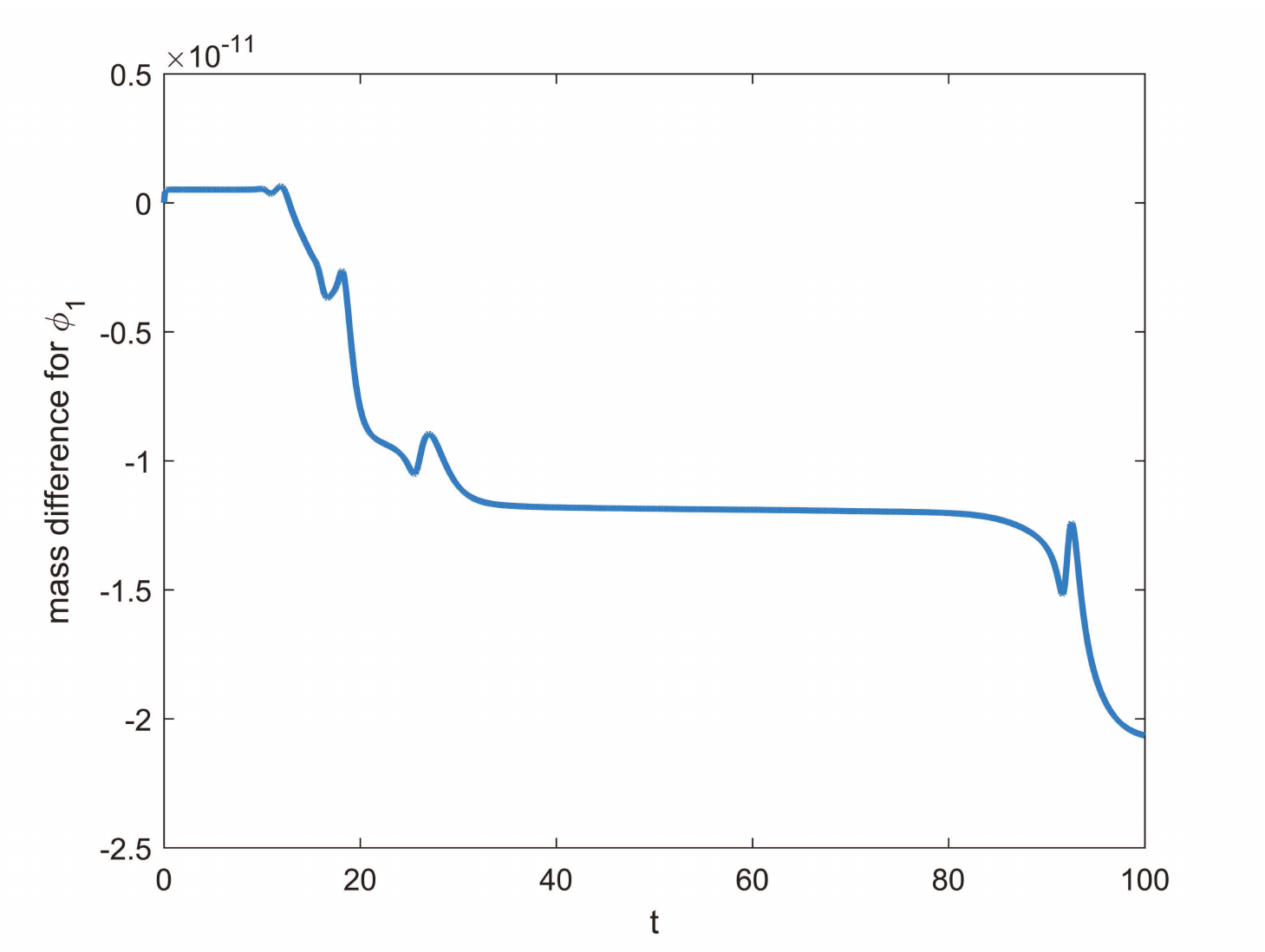}
		\end{subfigure}
		\begin{subfigure}{}
			\includegraphics[width=2.2in]{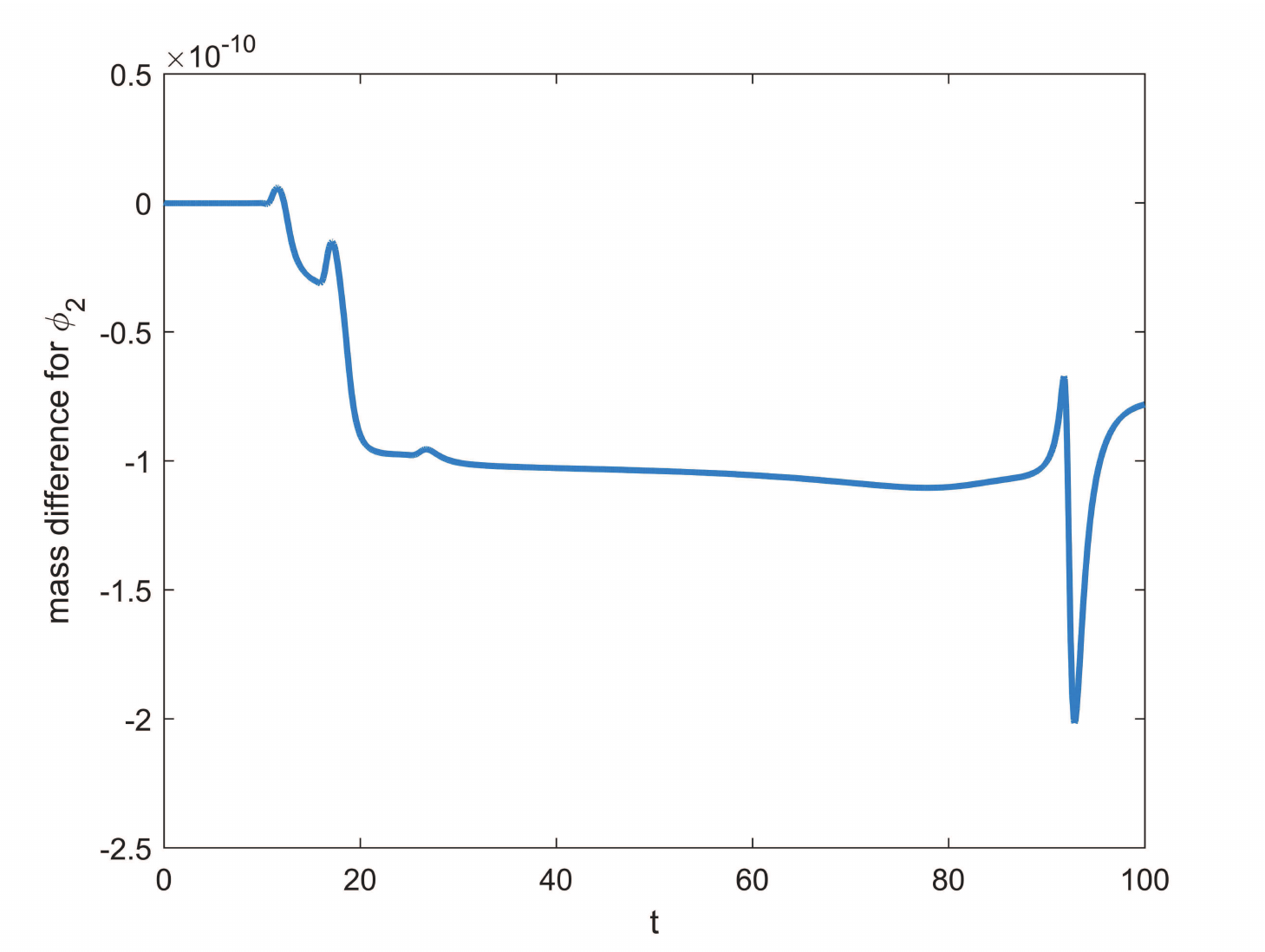}
		\end{subfigure}
	\end{center}
	\caption{Example \ref{example 1}: The error developments of the total mass for $\phi_1$ and $\phi_2$, respectively. }
	\label{fig:cosmass}
\end{figure}
\begin{figure}[!htp]
\begin{center}	
\includegraphics[width=4.5in]{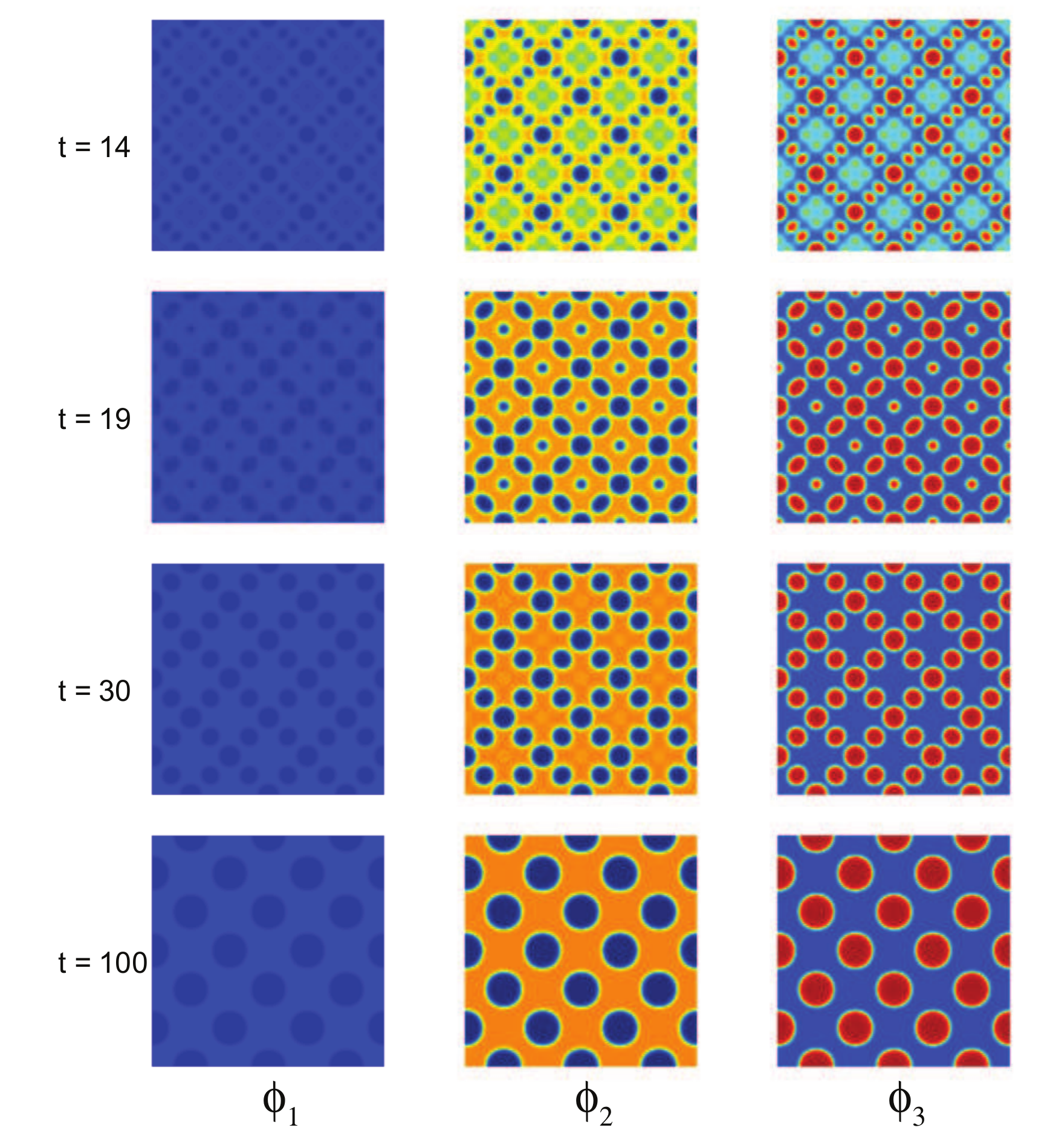}
	\caption{Example \ref{example 1}: Evolution of three phase variables at $t = 14, 19, 30$ and $100$. The time step size is taken as $\dt = 1.0 \times 10^{-4}$. }\label{fig:coslong}
\end{center}
\end{figure}
\begin{figure}[ht]
	\begin{center}
		\begin{subfigure}{}
			\includegraphics[width=2.2in]{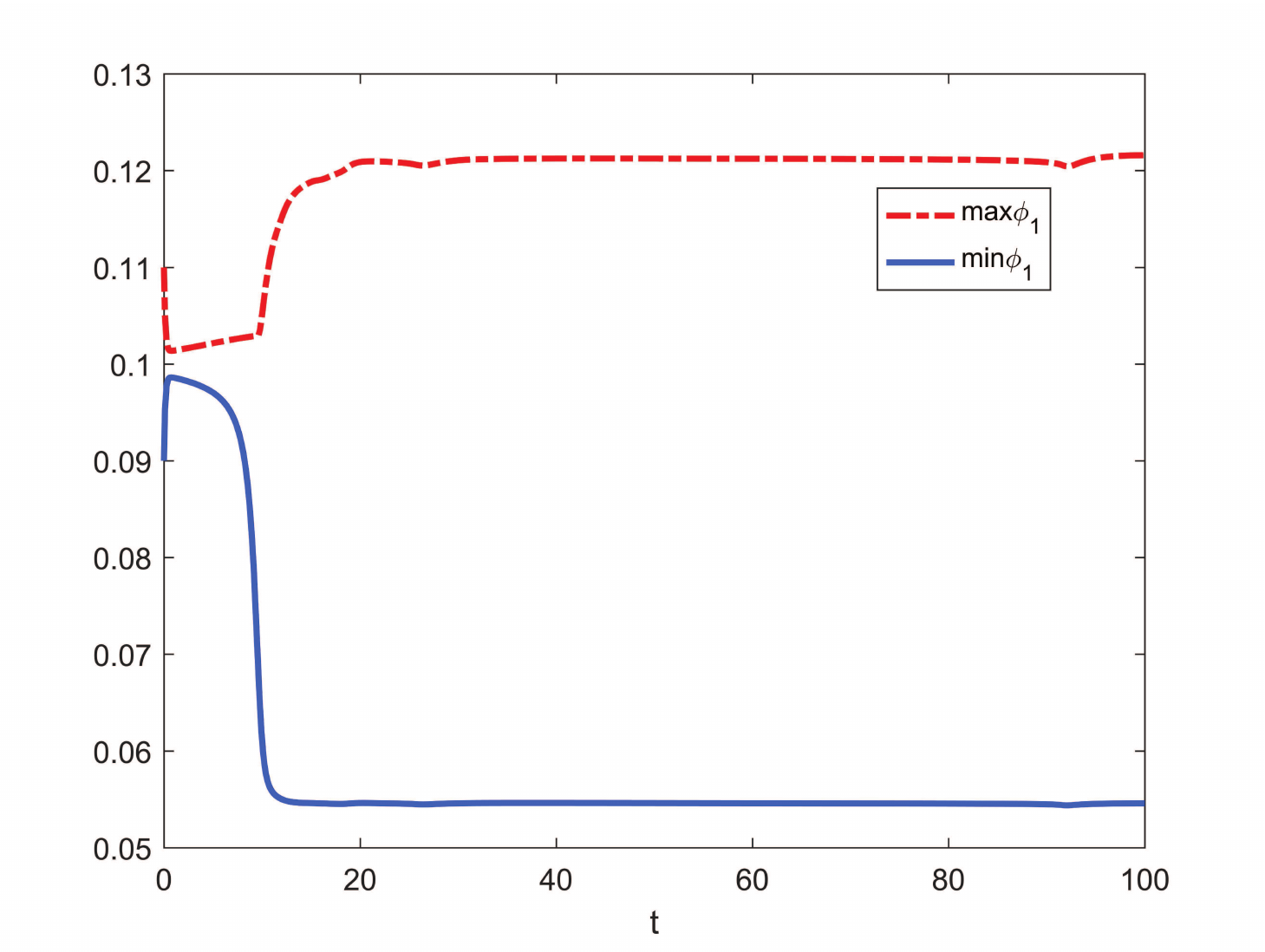}
		\end{subfigure}
		\begin{subfigure}{}
			\includegraphics[width=2.2in]{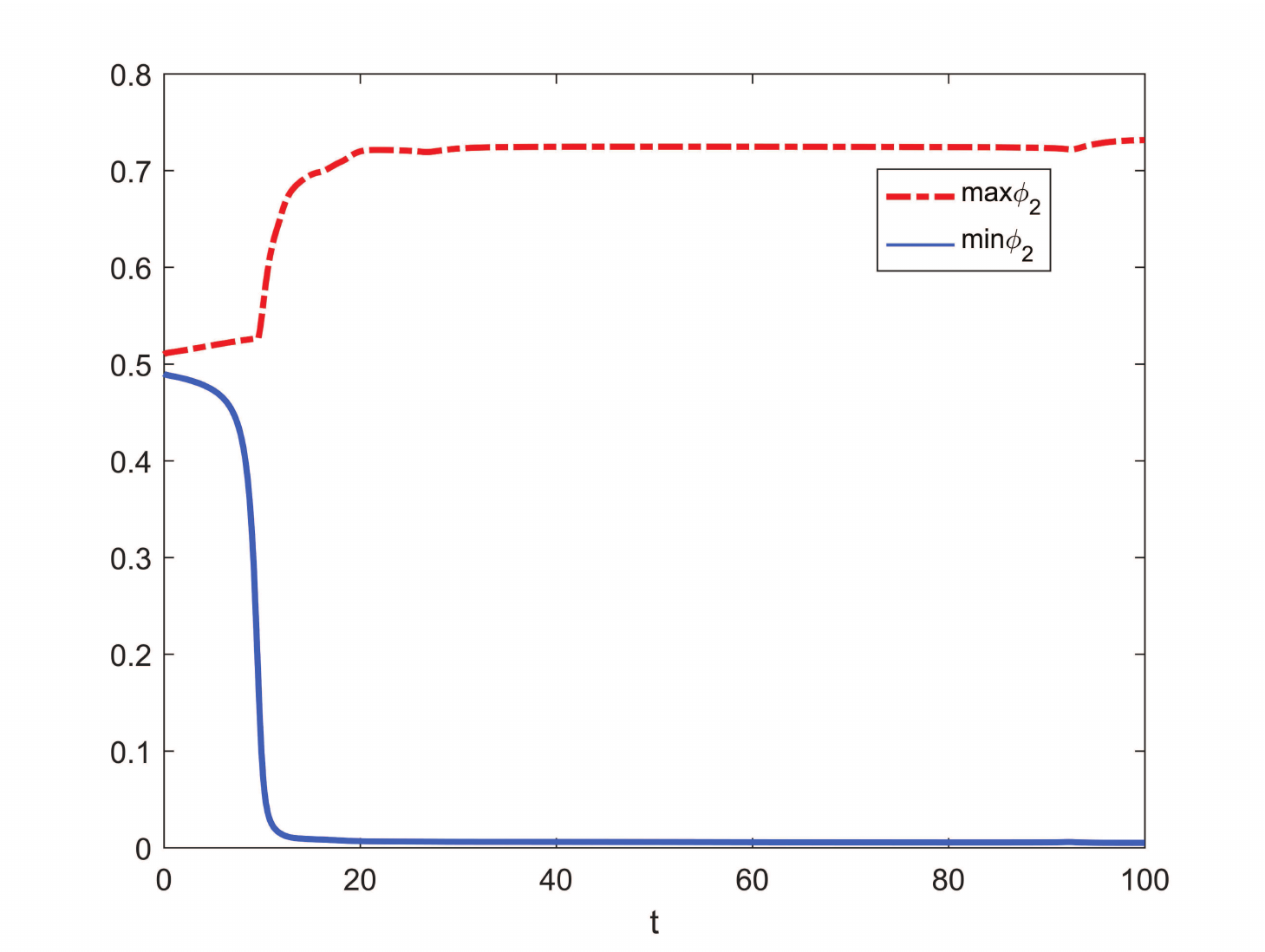}
		\end{subfigure}
	\end{center}
	\caption{Example \ref{example 1}: The time evolution of the maximum and minimum values for $\phi_1$ and $\phi_2$, respectively. }
	\label{fig:cosmaxmin}
\end{figure}
\begin{figure}[!htp]
\begin{center}	
\includegraphics[width=2.5in]{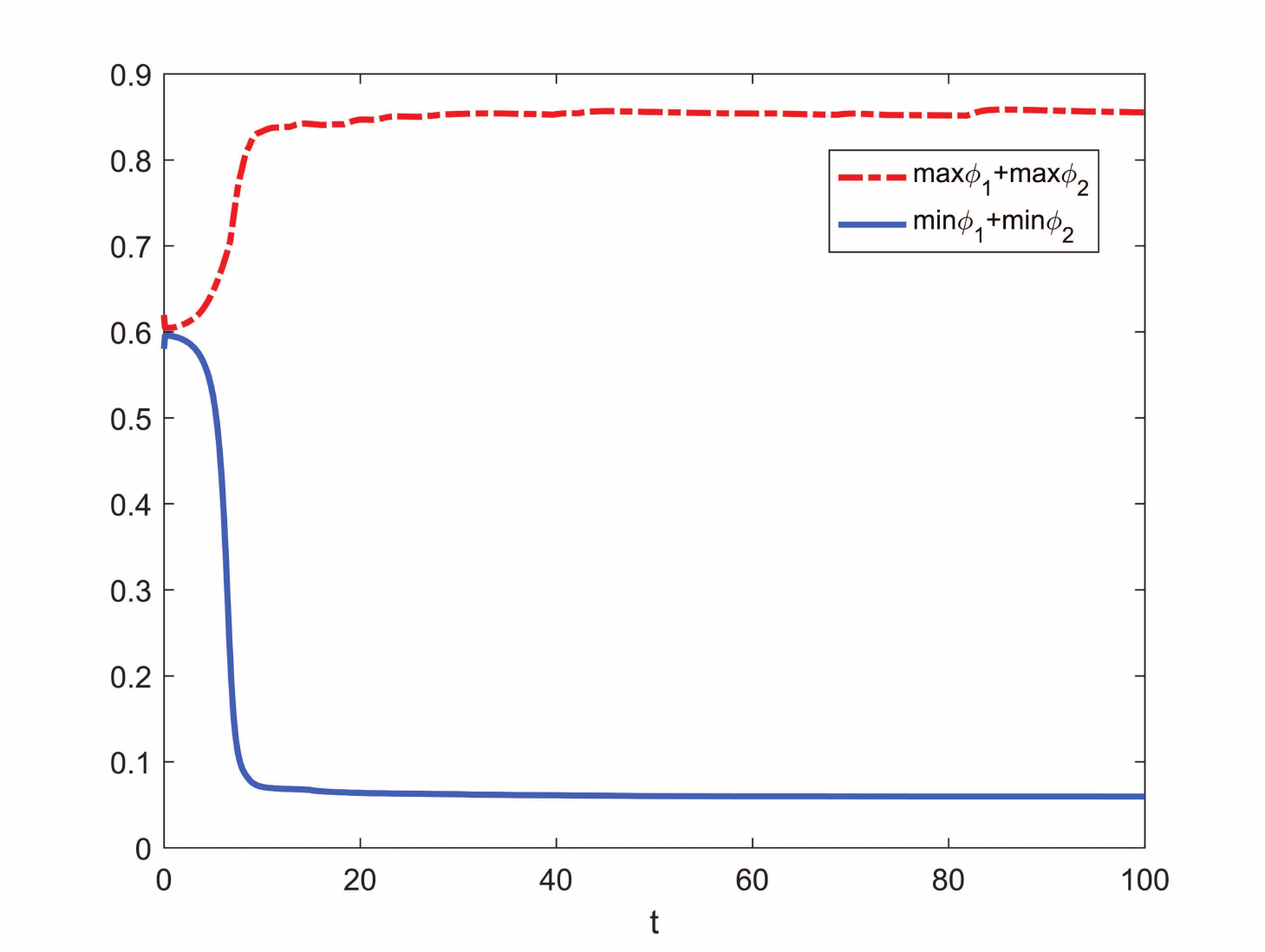}
	\caption{Example \ref{example 2}: The time evolutions of the maximum and minimum values for $\phi_1+\phi_2$, with $\dt = 1.0 \times 10^{-4}$.}\label{fig:coslong_maxminadd}
\end{center}
\end{figure}
\begin{example}\label{example 2}
A random initial perturbation is included in the initial data:
\begin{eqnarray}\label{eqn:init2}
\phi_1^0(x,y) = \phi_{10}+r_{i,j},\\
\phi_2^0(x,y) = \phi_{20}+r_{i,j},\nonumber
\end{eqnarray}
where the $r_{i,j}$ are uniformly distributed random numbers in [-0.01, 0.01].
\end{example}
This example is designed to study the influence of the different initial function and the statistical segment length on the phase transition of MMC hydrogels. We separately depict the phase states of the three variables, with four different $\varepsilon_i$ in \Cref{fig:randoma_i}. The
snapshot plots with four different $\phi_{10}$ and $\phi_{20}$ are presented in~\Cref{fig:randomphi1} and \Cref{fig:randomphi2}, respectively.
\begin{figure}[!htp]
\begin{center}	
\includegraphics[width=4.5in]{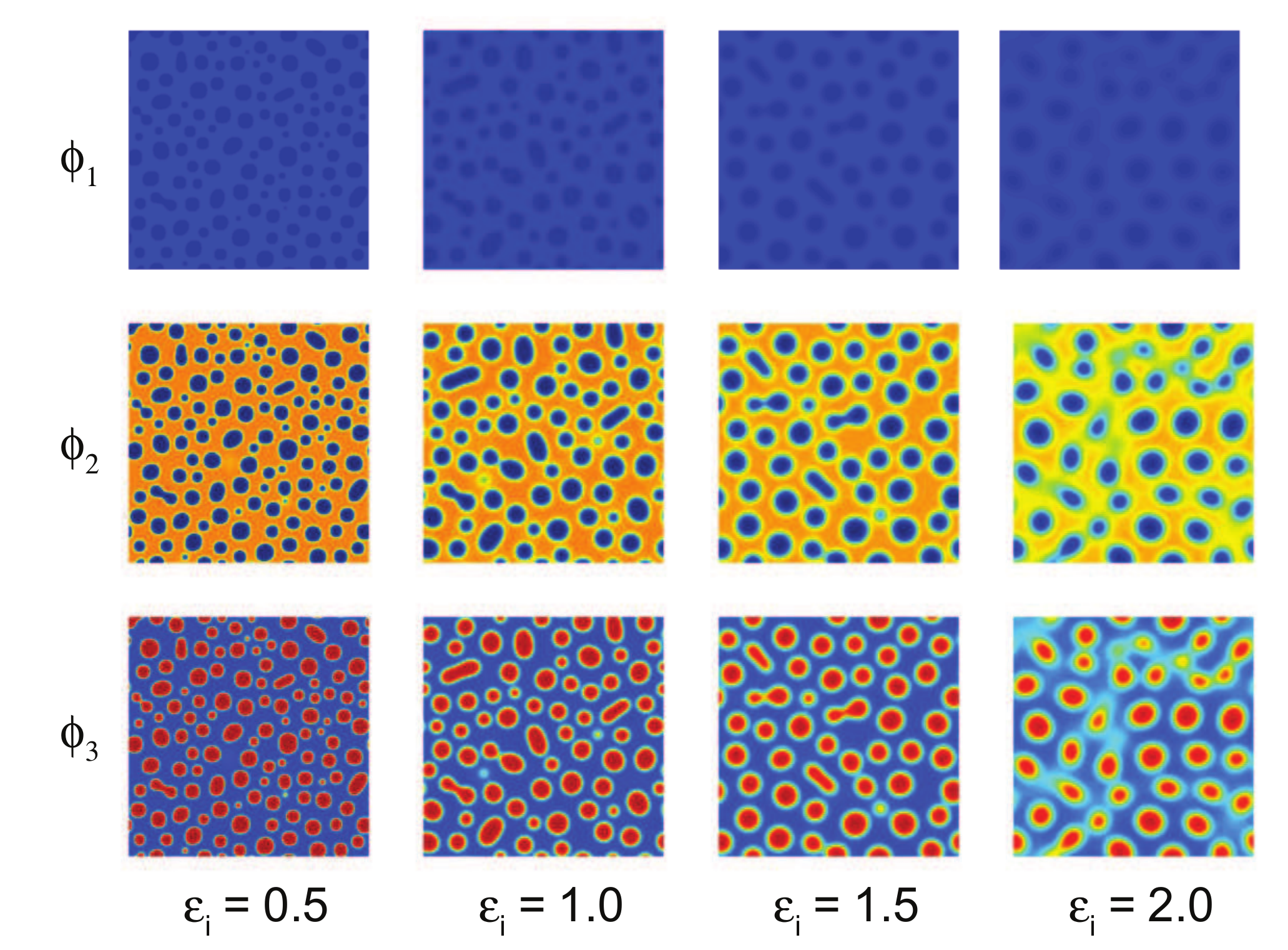}
	\caption{Example \ref{example 2}: The phase plots of three variables with different $\varepsilon_i$, $i = 1,2,3$ at $T=40$, and the time step size $\dt = 1.0 \times 10^{-4}$. }\label{fig:randoma_i}
\end{center}
\end{figure}
\begin{figure}[!htp]
\begin{center}	
\includegraphics[width=4.5in]{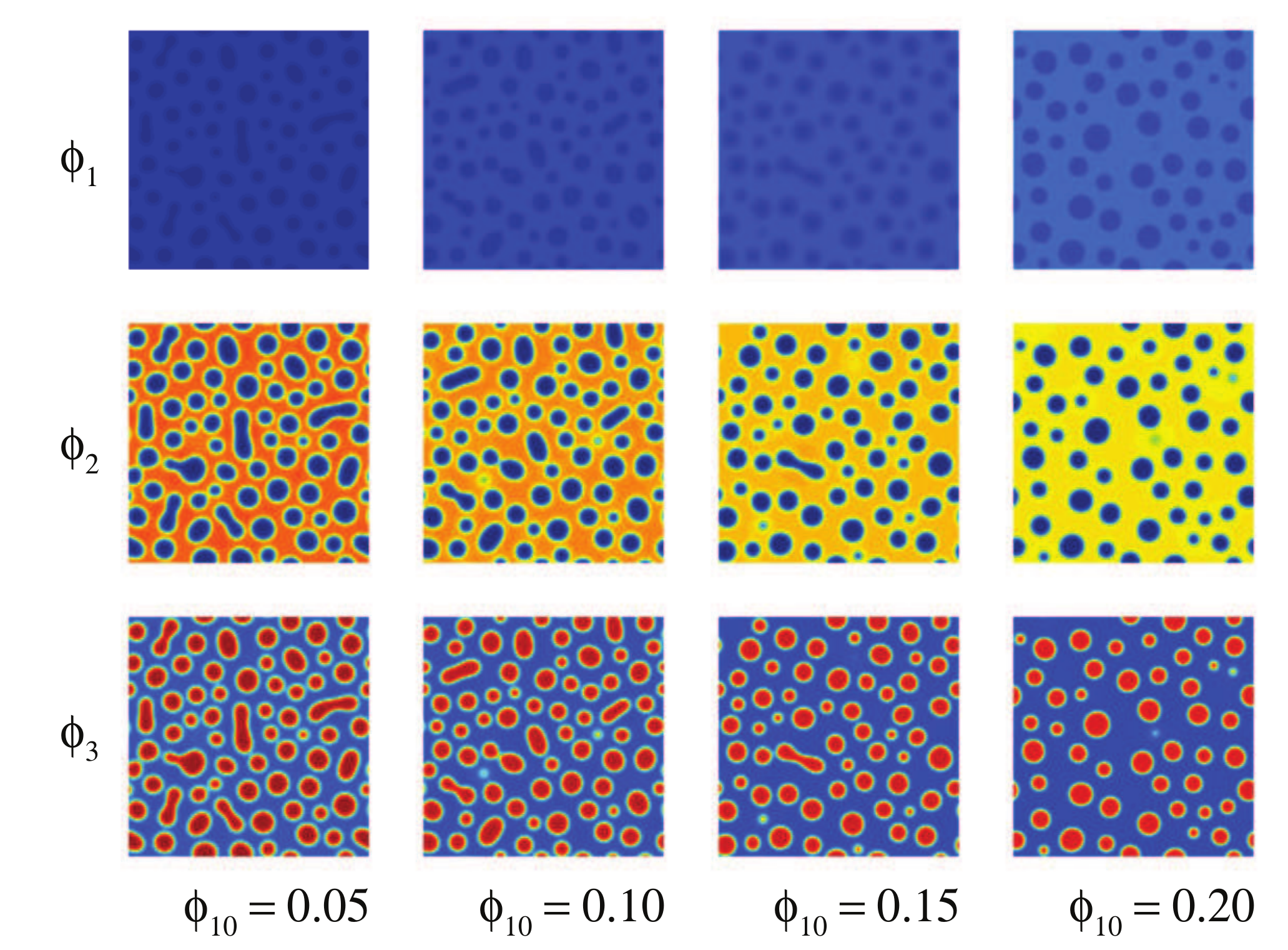}
	\caption{Example \ref{example 2}: The phase plots of three variables with different
$\phi_{10}$ at $T=40$, and the time step size $\dt = 1.0 \times 10^{-4}$. }\label{fig:randomphi1}
\end{center}
\end{figure}
\begin{figure}[!htp]
\begin{center}	
\includegraphics[width=4.5in]{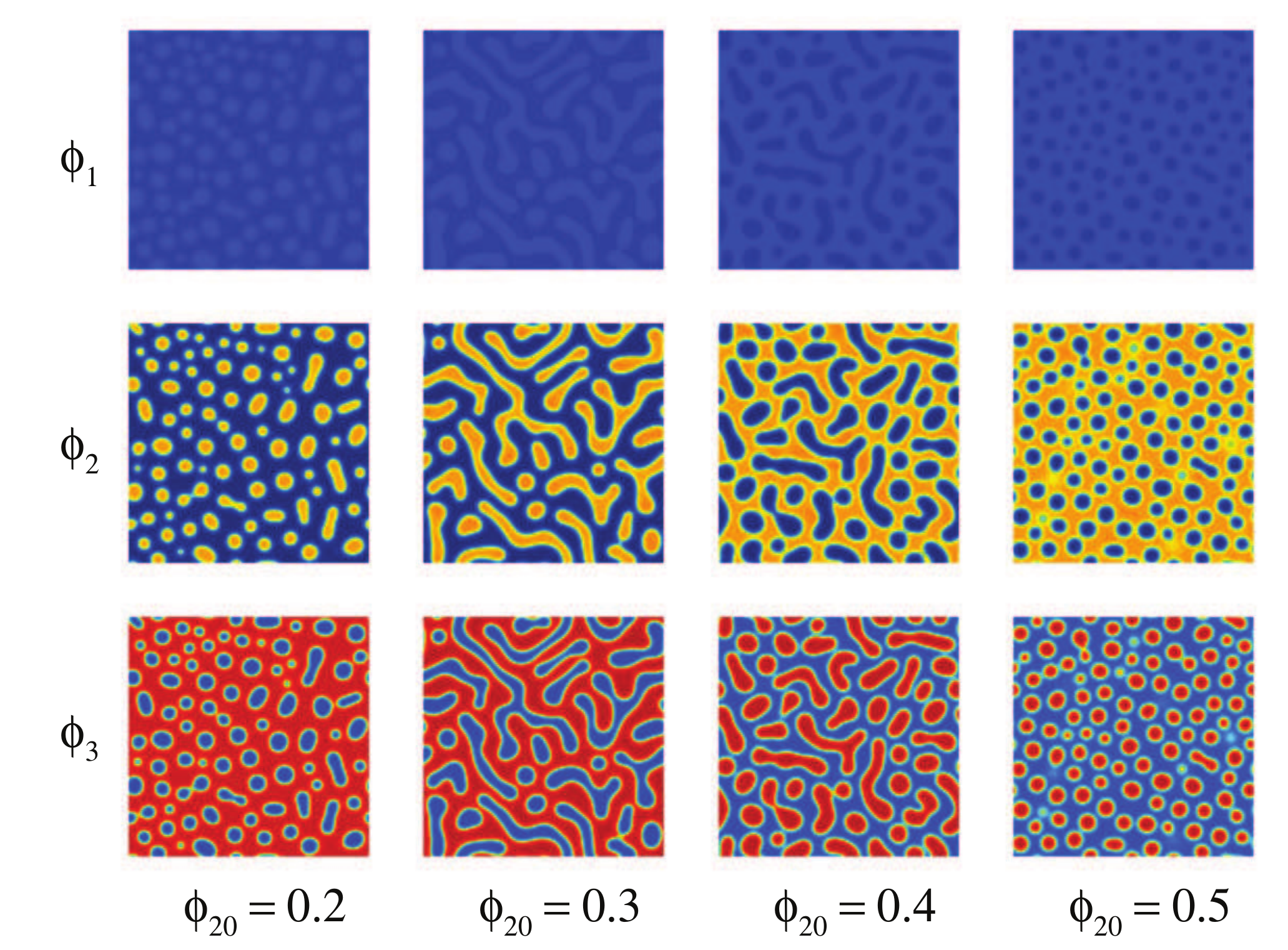}
	\caption{Example \ref{example 2}: The phase plots of three variables with different
$\phi_{20}$ at $T=20$, and the time step size $\dt = 1.0 \times 10^{-4}$. }\label{fig:randomphi2}
\end{center}
\end{figure}
\begin{example}\label{example 3}
The initial data is taken as:
\begin{eqnarray}\label{eqn:init3}
\phi_1^0(x,y) = 0.1+r_{i,j},\\
\phi_2^0(x,y) = 0.5+r_{i,j},\nonumber
\end{eqnarray}
where the $r_{i,j}$ are uniformly distributed random numbers in [-0.01, 0.01].
\end{example}
The energy evolution of the numerical solution (with $\dt = 1.0\times 10^{-4}$) is illustrated in~\Cref{fig:randomlong_energy}, which indicates an energy decay. In addition, we present the error evolution of the total mass of $\phi_1$ and $\phi_2$ in \cref{fig:randommass}. The maximum values and minimum values of $\phi_1$, $\phi_2$ and $\phi_1+\phi_2$ are displayed in \Cref{fig:randomlong_maxmin} and \Cref{fig:randomlong_maxminadd}. Moreover, in \Cref{fig:randomlong}, we plot the numerical solutions of $\phi_1$, $\phi_2$ and $\phi_3$ at a sequence of time instants to compare with the existing binary MMC results.
\begin{figure}[!htp]
\begin{center}	
\includegraphics[width=2.5in]{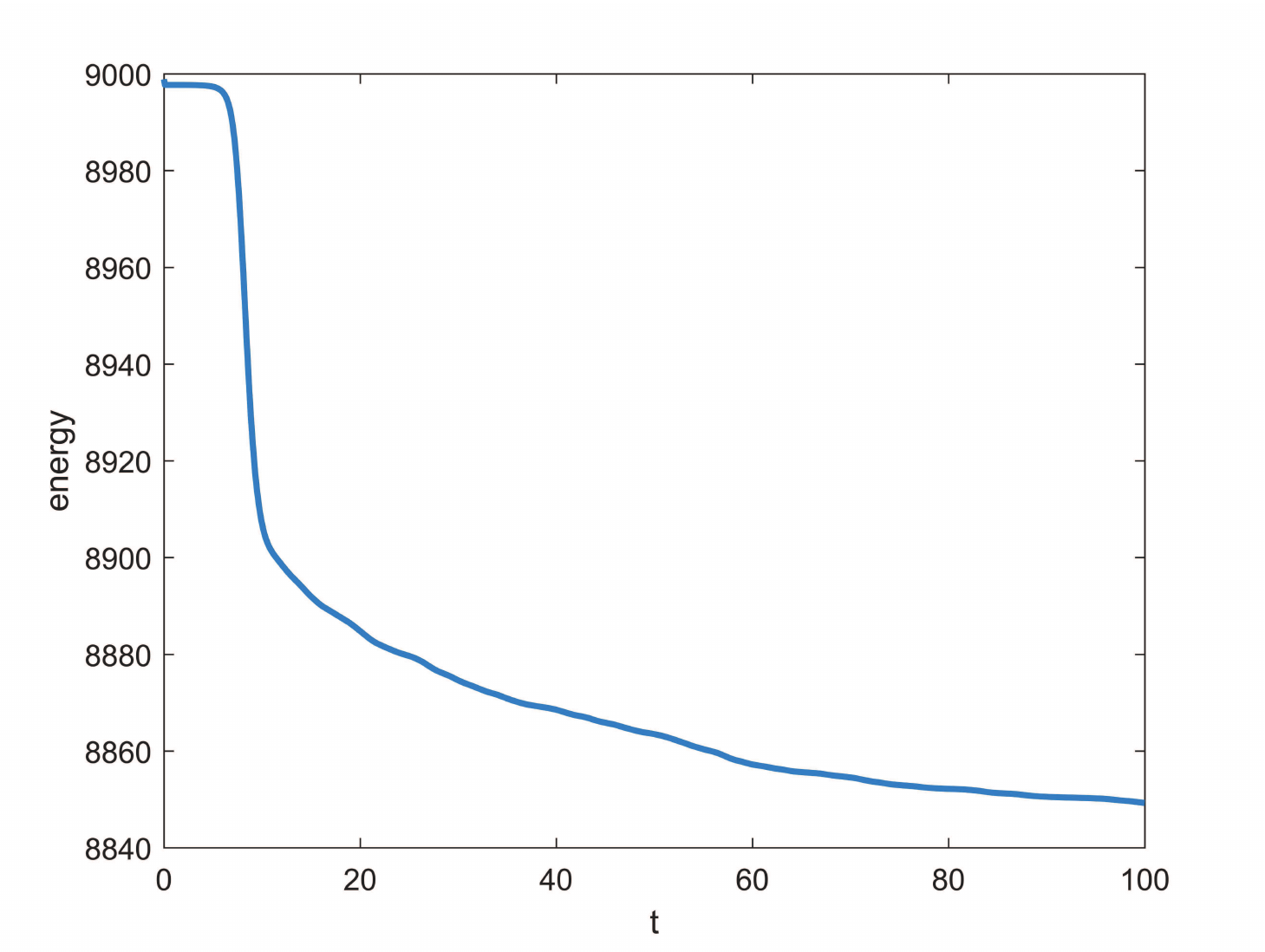}
	\caption{Example \ref{example 3}: Evolution of the energy over time, $\dt = 1.0 \times 10^{-4}$. }\label{fig:randomlong_energy}
\end{center}
\end{figure}
\begin{figure}[ht]
	\begin{center}
		\begin{subfigure}{}
			\includegraphics[width=2.2in]{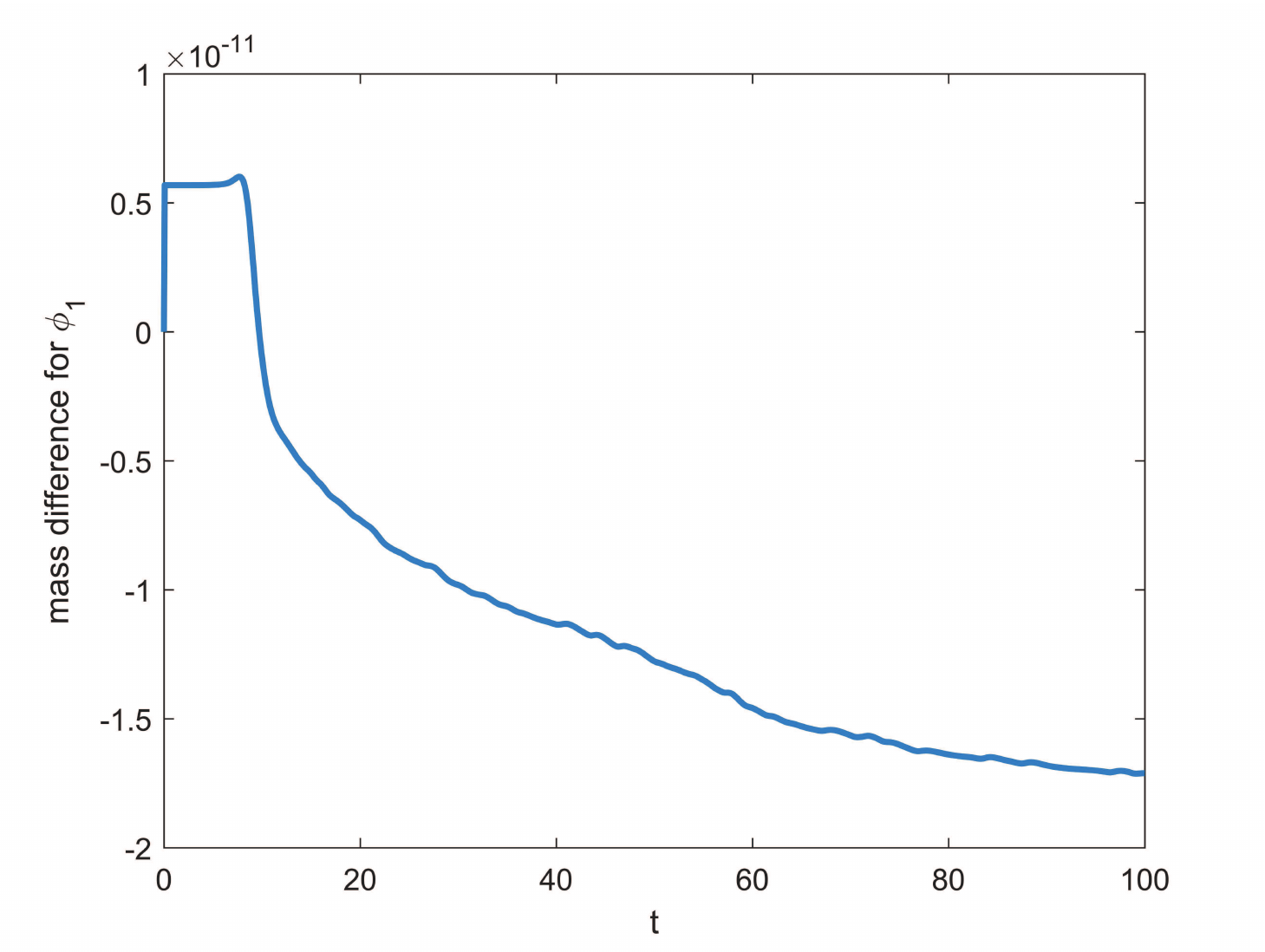}
		\end{subfigure}
		\begin{subfigure}{}
			\includegraphics[width=2.2in]{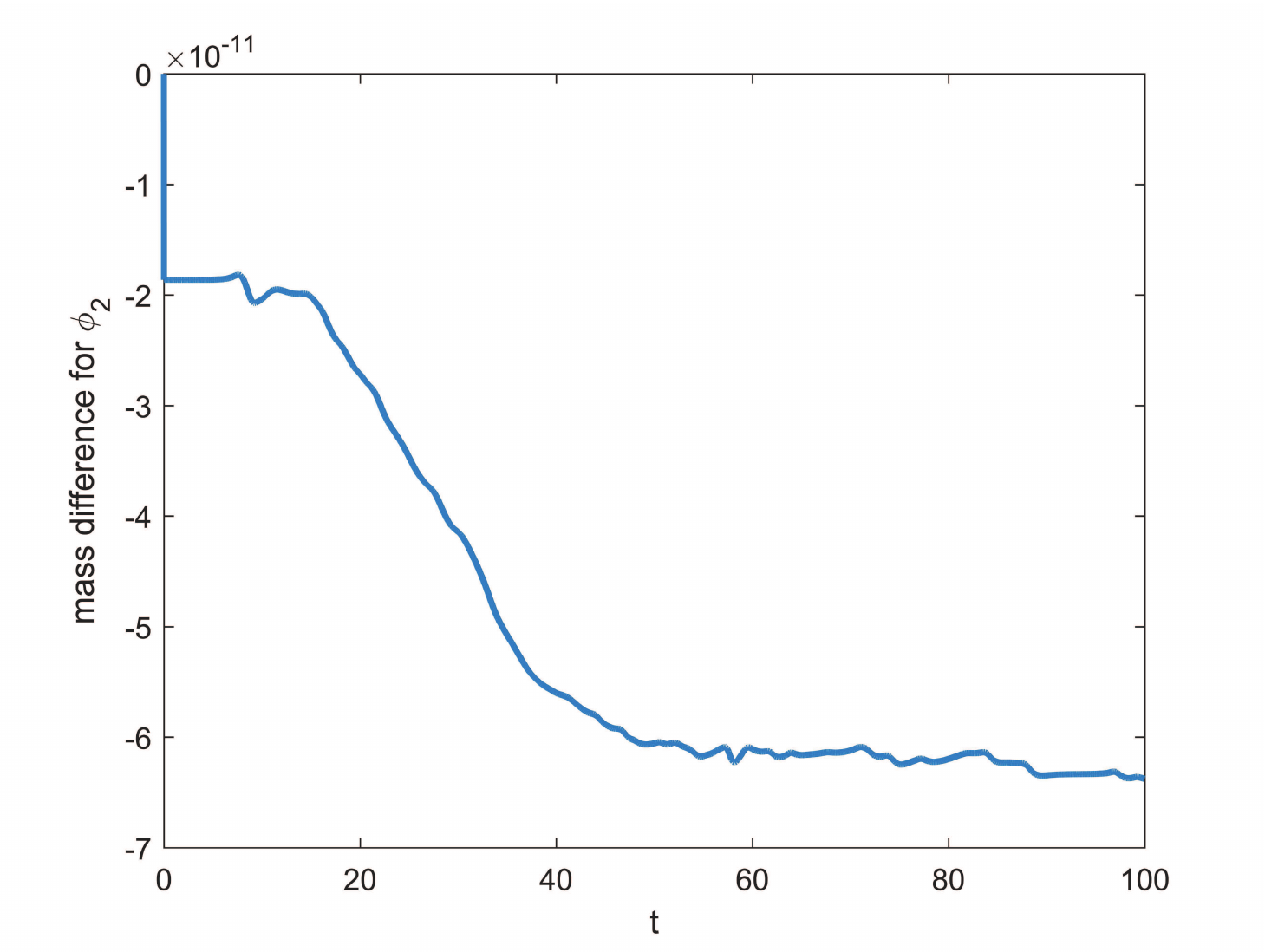}
		\end{subfigure}
	\end{center}
	 \caption{Example \ref{example 3}: The error development of the total mass for $\phi_1$ and $\phi_2$, respectively. }\label{fig:randommass}
\end{figure}
\begin{figure}[ht]
	\begin{center}
		\begin{subfigure}{}
			\includegraphics[width=2.2in]{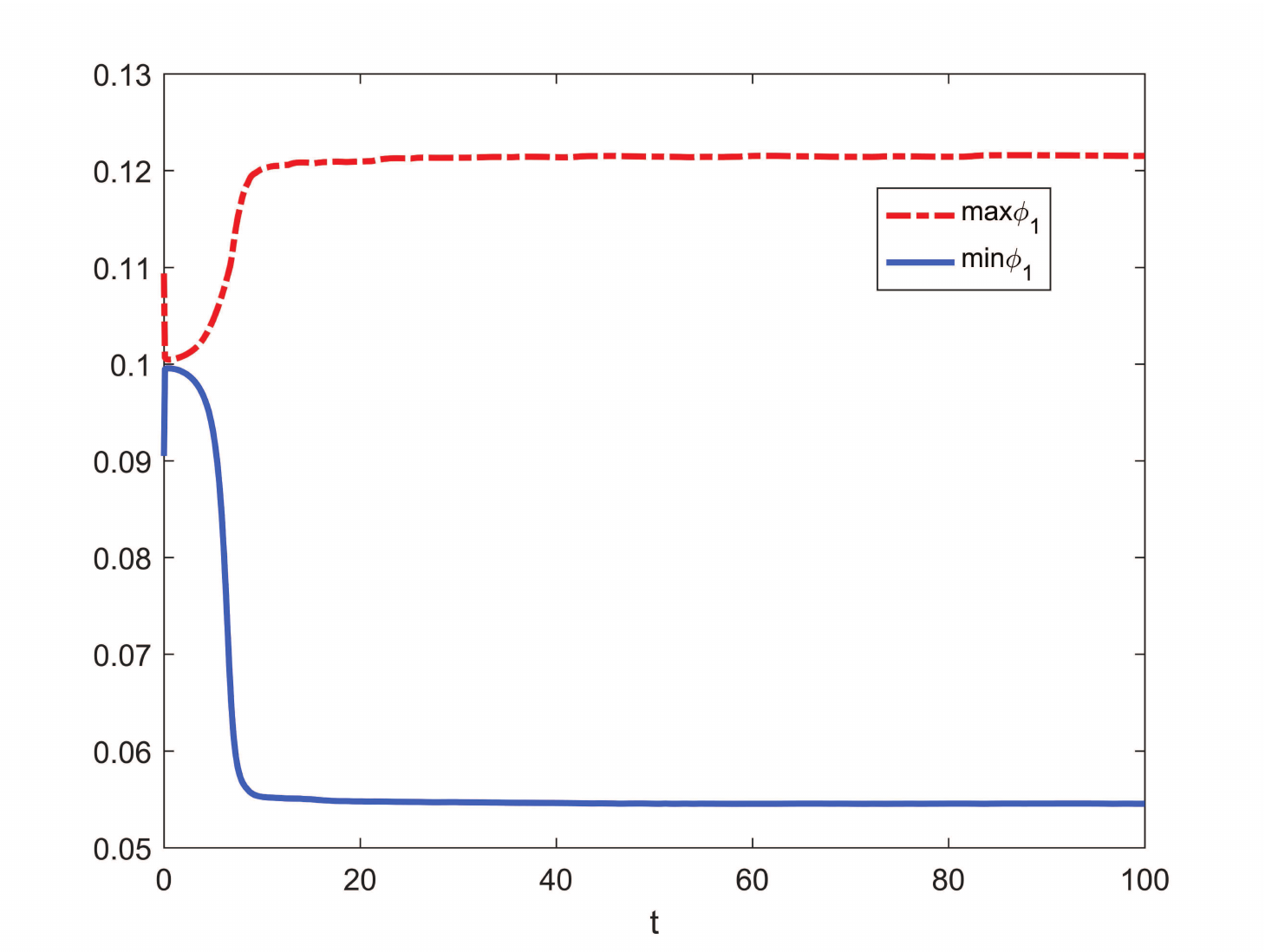}
		\end{subfigure}
		\begin{subfigure}{}
			\includegraphics[width=2.2in]{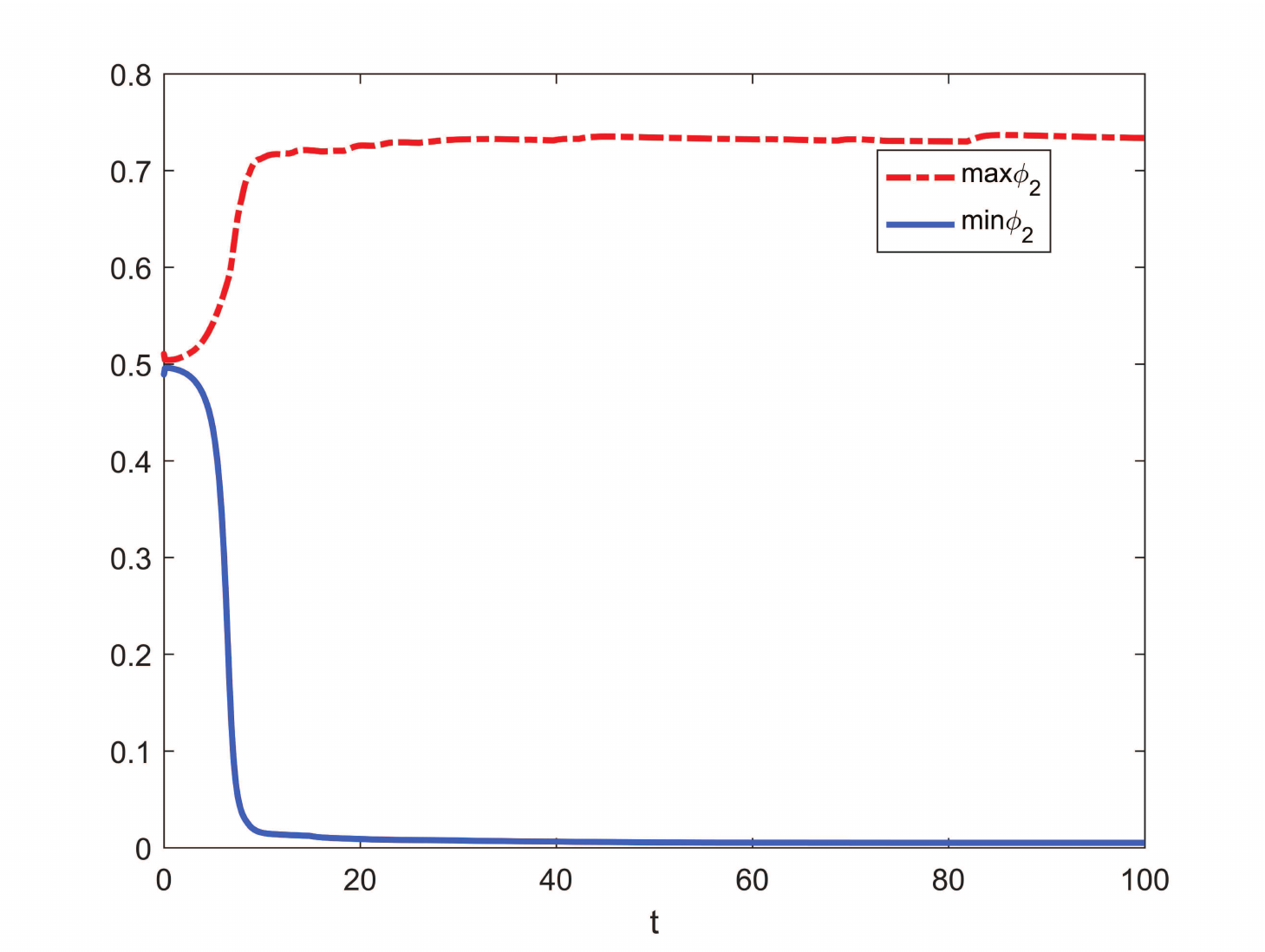}
		\end{subfigure}
	\end{center}
	\caption{Example \ref{example 3}:The time evolution of the maximum and minimum values for $\phi_1$ and $\phi_2$, respectively. }\label{fig:randomlong_maxmin}
\end{figure}
\begin{figure}[!htp]
\begin{center}	
\includegraphics[width=2.5in]{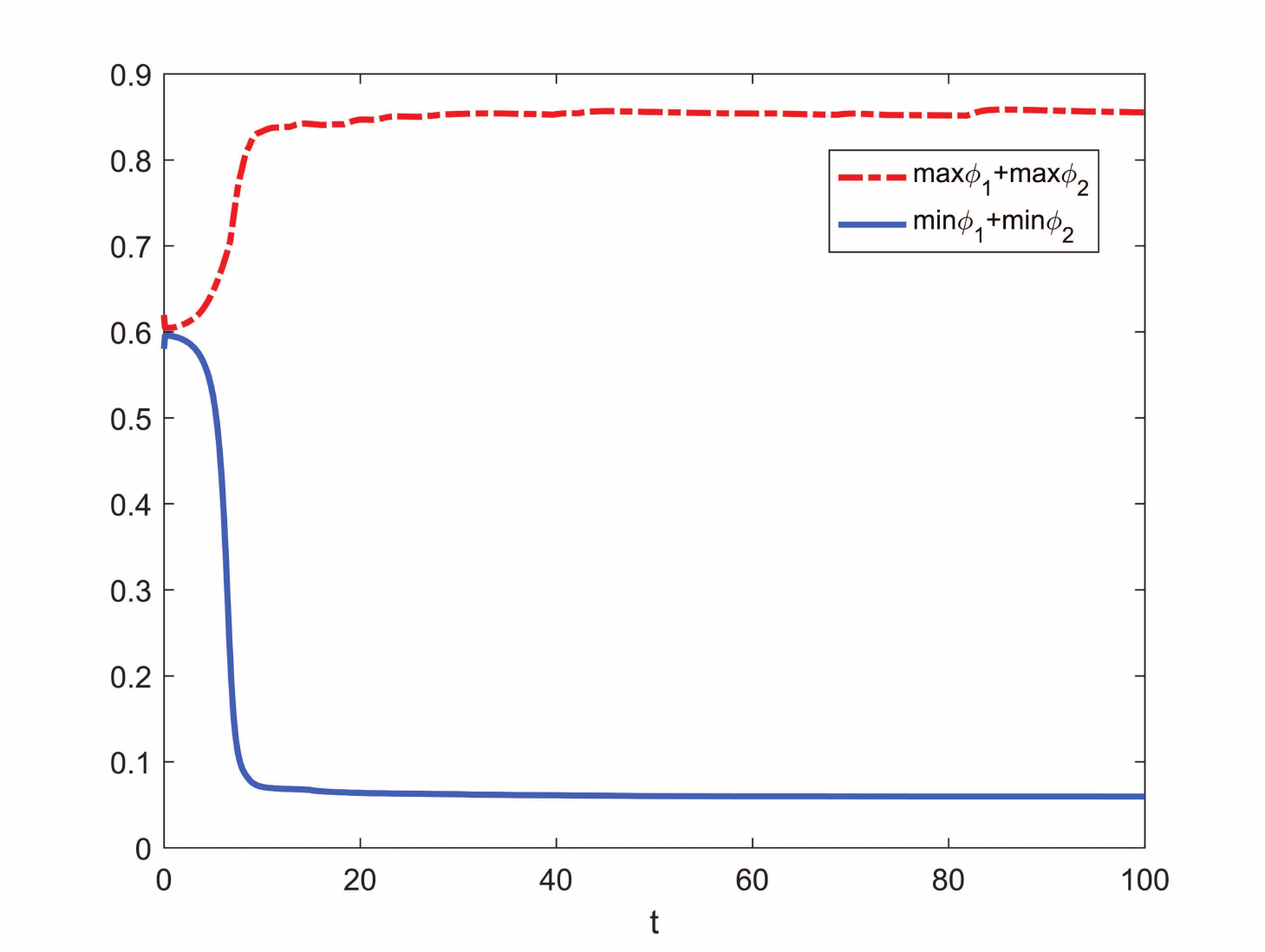}
	\caption{Example \ref{example 3}: The time evolution of the maximum and minimum values for $\phi_1+\phi_2$.}\label{fig:randomlong_maxminadd}
\end{center}
\end{figure}
\begin{figure}[!htp]
\begin{center}	
\includegraphics[width=4.5in]{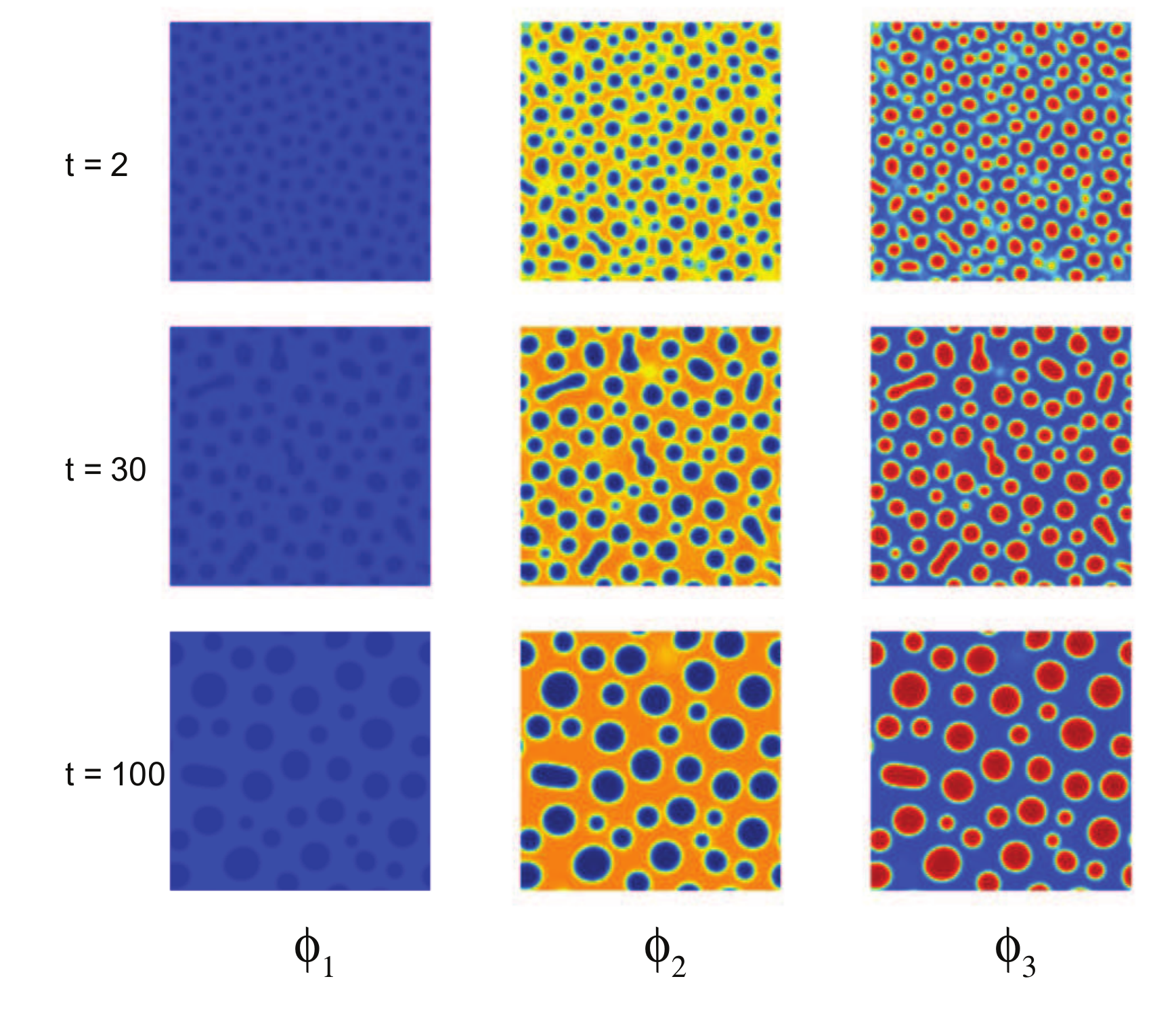}
	\caption{Example \ref{example 3}: Evolution of the three phase variables with at $t = 2, 30$ and $100$, with $\dt = 1.0 \times 10^{-4}$. }\label{fig:randomlong}
\end{center}
\end{figure}
%
\section{Concluding remarks}
\label{sec:conclusions}
In this paper, we develop a uniquely solvable, positivity preserving and unconditionally energy stable finite difference scheme for the ternary Cahn-Hilliard-like model, describing the dynamics of the MMC hydrogels system. The free energy functional of the phase model is of Flory-Huggins-deGennes type, dependent on three variables, which could be reduced to two variables by the total mass identity. The numerical scheme is designed based on the convex-concave decomposition of the physical energy, which is highly non-standard due to the multi phase variables involved. A theoretical justification of the positivity-preserving property has been established, by constructing a strictly convex discrete energy functional in two variables via the mass conservation identity, combined with the following subtle fact: the singular feature of the logarithmic functions ensures that a pair of minimizers could not occur on the limit values. In addition, the appearance of the highly nonlinear and singular coefficients in the surface diffusion part has also ensured the positivity-preserving property. The unique solvability and unconditional energy stability come from the convexity analysis. The FAS nonlinear multigrid method and Newton iteration algorithm are employed to improve the efficiency in the practical computation. A few numerical results have also been presented to demonstrate the robustness of the proposed scheme.

\section*{Acknowledgments}
L.X.~Dong is supported by the China Scholarship Council (CSC) during her visit in the University of Tennessee, Knoxville. Z.R.~Zhang is partlialy supported by the National Natural Science Foundation of China~(NSFC) No.11871105 and Science Challenge Project No. TZ2018002. C.~Wang is partially supported by the NSF DMS-2012669, S.M. Wise is partially supported by the NSF DMS-1719854, DMS-2012634.  

\bibliographystyle{siamplain}
\bibliography{MMC3term}

\begin{thebibliography}{10}

\bibitem{baskaran13a}
{\sc A.~Baskaran, Z.~Hu, J.~Lowengrub, C.~Wang, S.~Wise, and P.~Zhou}, {\em
  Energy stable and efficient finite-difference nonlinear multigrid schemes for
  the modified phase field crystal equation}, J. Comput. Phys., 250 (2013),
  pp.~270--292.

\bibitem{Boyer2006Numerical}
{\sc F.~Boyer and C.~Lapuerta}, {\em Study of a three component {Cahn-Hilliard}
  flow model}, M2AN Math. Model. Numer. Anal., 40 (2006), pp.~653--687.

\bibitem{Boyer2011Numerical}
{\sc F.~Boyer and S.~Minjeaud}, {\em Numerical schemes for a three component
  {Cahn-Hilliard} model}, M2AN Math. Model. Numer. Anal., 45 (2011),
  pp.~697--738.

\bibitem{chen19a}
{\sc W.~Chen, W.~Feng, Y.~Liu, C.~Wang, and S.~Wise}, {\em A second order
  energy stable scheme for the {Cahn-Hilliard-Hele-Shaw} equation}, Disc. Cont.
  Dyn. Sys. B, 24 (2019), pp.~149--182.

\bibitem{chen16}
{\sc W.~Chen, Y.~Liu, C.~Wang, and S.~Wise}, {\em An optimal-rate convergence
  analysis of a fully discrete finite difference scheme for
  {Cahn-Hilliard-Hele-Shaw} equation}, Math. Comp.,, 85 (2016), pp.~2231--2257.

\bibitem{ChenTernaryCH}
{\sc W.~Chen, C.~Wang, S.~Wang, X.~Wang, and S.~Wise}, {\em Energy stable
  numerical schemes for a ternary {Cahn-Hilliard} system}, J. Sci. Comput.,
  (2020).
\newblock Submitted and in review.

\bibitem{chen19b}
{\sc W.~Chen, C.~Wang, X.~Wang, and S.~Wise}, {\em Positivity-preserving,
  energy stable numerical schemes for the {Cahn-Hilliard} equation with
  logarithmic potential}, J. Comput. Phys. X, 3 (2019), p.~100031.

\bibitem{chenY18}
{\sc Y.~Chen, J.~Lowengrub, J.~Shen, C.~Wang, and S.~Wise}, {\em Efficient
  energy stable schemes for isotropic and strongly anisotropic {Cahn-Hilliard}
  systems with the {Willmore} regularization}, J. Comput. Phys., 365 (2018),
  pp.~57--73.

\bibitem{cheng2019a}
{\sc K.~Cheng, W.~Feng, C.~Wang, and S.~Wise}, {\em An energy stable fourth
  order finite difference scheme for the {Cahn-Hilliard} equation}, J. Comput.
  Appl. Math., 362 (2019), pp.~574--595.

\bibitem{cheng16a}
{\sc K.~Cheng, C.~Wang, S.~Wise, and X.~Yue}, {\em A second-order, weakly
  energy-stable pseudo-spectral scheme for the {Cahn-Hilliard} equation and its
  solution by the homogeneous linear iteration method}, J. Sci. Comput., 69
  (2016), pp.~1083--1114.

\bibitem{diegel15a}
{\sc A.~Diegel, X.~Feng, and S.~Wise}, {\em Convergence analysis of an
  unconditionally stable method for a {Cahn-Hilliard-Stokes} system of
  equations}, SIAM J. Numer. Anal., 53 (2015), pp.~127--152.

\bibitem{diegel17}
{\sc A.~Diegel, C.~Wang, X.~Wang, and S.~Wise}, {\em Convergence analysis and
  error estimates for a second order accurate finite element method for the
  {Cahn-Hilliard-Navier-Stokes} system}, Numer. Math., 137 (2017),
  pp.~495--534.

\bibitem{diegel16}
{\sc A.~Diegel, C.~Wang, and S.~Wise}, {\em Stability and convergence of a
  second order mixed finite element method for the {Cahn-Hilliard} equation},
  IMA J. Numer. Anal., 36 (2016), pp.~1867--1897.

\bibitem{Dong2018Convergence}
{\sc L.~Dong, W.~Feng, C.~Wang, S.~Wise, and Z.~Zhang}, {\em Convergence
  analysis and numerical implementation of a second order numerical scheme for
  the three-dimensional phase field crystal equation}, Comput. Math. Appl., 75
  (2018), pp.~1912--1928.

\bibitem{Dong2019a}
{\sc L.~Dong, C.~Wang, H.~Zhang, and Z.~Zhang}, {\em A positivity-preserving,
  energy stable and convergent numerical scheme for the {Cahn-Hilliard}
  equation with a {Flory-Huggins-deGennes} energy}, Commun. Math. Sci., 17
  (2019), pp.~921--939.

\bibitem{Dong2020a}
{\sc L.~Dong, C.~Wang, H.~Zhang, and Z.~Zhang}, {\em A positivity-preserving
  second-order {BDF} scheme for the {Cahn-Hilliard} equation with variable
  interfacial parameters}, Commun. Comput. Phys., 28 (2020), pp.~967--998.
\newblock Accepted and in press.

\bibitem{feng2018bsam}
{\sc W.~Feng, Z.~Guo, J.~Lowengrub, and S.~Wise}, {\em A mass-conservative
  adaptive fas multigrid solver for cell-centered finite difference methods on
  block-structured,locally-cartesian grids}, J. Comput. Phys., 352 (2018),
  pp.~463--497.

\bibitem{fengW18b}
{\sc W.~Feng, C.~Wang, S.~Wise, and Z.~Zhang}, {\em A second-order energy
  stable {Backward Differentiation Formula} method for the epitaxial thin film
  equation with slope selection}, Numer. Methods Partial Differ. Equ., 34
  (2018), pp.~1975--2007.

\bibitem{guillen14}
{\sc F.~Guill{\'e}n-Gonz{\'a}lez and G.~Tierra}, {\em Second order schemes and
  time-step adaptivity for {Allen-Cahn} and {Cahn-Hilliard} models}, Comput.
  Math. Appl., 68 (2014), pp.~821--846.

\bibitem{guo16}
{\sc J.~Guo, C.~Wang, and S.~Wise}, {\em An {$H^2$} convergence of a
  second-order convex-splitting, finite difference scheme for the
  three-dimensional {Cahn-Hilliard} equation}, Commu. Math. Sci., 14 (2016),
  pp.~489--515.

\bibitem{han15}
{\sc D.~Han and X.~Wang}, {\em A second order in time, uniquely solvable,
  unconditionally stable numerical scheme for {Cahn-Hilliard-Navier-Stokes}
  equation}, J. Comput. Phys., 290 (2015), pp.~139--156.

\bibitem{hu09}
{\sc Z.~Hu, S.~Wise, C.~Wang, and J.~Lowengrub}, {\em Stable and efficient
  finite-difference nonlinear-multigrid schemes for the phase-field crystal
  equation}, J. Comput. Phys., 228 (2009), pp.~5323--5339.

\bibitem{LiD2017}
{\sc D.~Li and Z.~Qiao}, {\em On second order semi-implicit fourier spectral
  methods for {2D Cahn-Hilliard} equations}, J. Sci. Comput., 70 (2017),
  pp.~301--341.

\bibitem{LiD2016a}
{\sc D.~Li, Z.~Qiao, and T.~Tang}, {\em Characterizing the stabilization size
  for semi-implicit {Fourier-spectral} method to phase field equations}, SIAM
  J. Numer. Anal., 54 (2016), pp.~1653--1681.

\bibitem{Lixiao2015}
{\sc X.~Li, G.~Ji, and H.~Zhang}, {\em Phase transitions of macromolecular
  microsphere composite hydrogels based on the stochastic {Cahn-Hilliard}
  equation}, J. Comput. Phys., 283 (2015), pp.~81--97.

\bibitem{Li2016An}
{\sc X.~Li, Z.~Qiao, and H.~Zhang}, {\em An unconditionally energy stable
  finite difference scheme for a stochastic {Cahn-Hilliard} equation}, Sci.
  China. Math., 59 (2016), pp.~1815--1834.

\bibitem{Lixiao2}
{\sc X.~Li, Z.~Qiao, and H.~Zhang}, {\em A second-order convex-splitting scheme
  for a {Cahn-Hilliard} equation with variable interfacial parameters}, J.
  Comput. Math., 35 (2017), pp.~693--710.

\bibitem{liuY17}
{\sc Y.~Liu, W.~Chen, C.~Wang, and S.~Wise}, {\em Error analysis of a mixed
  finite element method for a {Cahn-Hilliard-Hele-Shaw} system}, Numer. Math.,
  135 (2017), pp.~679--709.

\bibitem{shen18b}
{\sc J.~Shen and J.~Xu}, {\em Convergence and error analysis for the scalar
  auxiliary variable {(SAV)} schemes to gradient flows}, SIAM J. Numer. Anal.,
  56 (2018), pp.~2895--2912.

\bibitem{shen18a}
{\sc J.~Shen, J.~Xu, and J.~Yang}, {\em The scalar auxiliary variable (sav)
  approach for gradient flows}, J. Comput. Phys., 353 (2018), pp.~407--416.

\bibitem{wang11a}
{\sc C.~Wang and S.~Wise}, {\em An energy stable and convergent
  finite-difference scheme for the modified phase field crystal equation}, SIAM
  J. Numer. Anal., 49 (2011), pp.~945--969.

\bibitem{wise10}
{\sc S.~Wise}, {\em Unconditionally stable finite difference, nonlinear
  multigrid simulation of the {Cahn-Hilliard-Hele-Shaw} system of equations},
  J. Sci. Comput., 44 (2010), pp.~38--68.

\bibitem{wise09a}
{\sc S.~Wise, C.~Wang, and J.~Lowengrub}, {\em An energy stable and convergent
  finite-difference scheme for the phase field crystal equation}, SIAM J.
  Numer. Anal., 47 (2009), pp.~2269--2288.

\bibitem{Yan18}
{\sc Y.~Yan, W.~Chen, C.~Wang, and S.~Wise}, {\em A second-order energy stable
  bdf numerical scheme for the cahn-hilliard equation}, Commun. Comput. Phys.,
  23 (2018), pp.~572--602.

\bibitem{Yang2017Numerical}
{\sc X.~Yang, J.~Zhao, Q.~Wang, and J.~Shen}, {\em Numerical approximations for
  a three-components {Cahn–Hilliard} phase-field model based on the invariant
  energy quadratization method}, Math. Models Methods Appl. Sci.,  (2017),
  pp.~1--38.

\bibitem{Zhai2012Investigation}
{\sc D.~Zhai and H.~Zhang}, {\em Investigation on the application of the tdgl
  equation in macromolecular microsphere composite hydrogel}, Soft Matter, 9
  (2012), pp.~820--825.

\end{thebibliography}
\end{document}